\documentclass[a4paper,12pt]{article}

\usepackage{pstricks}
\usepackage{graphicx,psfrag}
\usepackage{amsmath}
\usepackage{amssymb}
\usepackage{amsthm}
\usepackage{color}

      \newcommand {\al}   {\alpha}          \newcommand {\bt}  {\beta}
      \newcommand {\gam } {\gamma}          \newcommand {\Gam}  {\Gamma}
      \newcommand {\del}  {\delta}          \newcommand {\Del} {\Delta}
              \newcommand {\ve}   {\varepsilon}
                 \newcommand {\vphi} {\varphi}
      \newcommand {\lam}  {\lambda}         
               \newcommand {\rrr}   {r}     \newcommand {\vvv}   {v}
      \newcommand {\eee}   {e}       \newcommand {\nnn}   {n}
                \newcommand {\Om}  {\Omega}
      \newcommand {\pl}   {\partial}        
          
           \newcommand {\UUU}  {{\cal U}}
      \newcommand {\HHH}    {\mathcal{H}}       
      \newcommand {\RRR}  {{\mathbb R}}     
      \newcommand {\OOO}  {{\cal O}}

              \newcommand {\BBB}  {{\cal B}}
           \newcommand {\NNN}  {\mathcal{N}}
       
     \newcommand {\beq}  {\begin{equation}}
      \newcommand {\eeq}  {\end{equation}}

      \newtheorem{theorem}{Theorem}
      \newtheorem{lemma}{Lemma}
      \newtheorem{utv}{Proposition}
      \newtheorem{zam}{Remark}

      \newtheorem{primer}{Example}
          
                \newtheorem{corollary}{Corollary}

\title{Behavior of convex surfaces near ridge points}

\author{Alexander Plakhov\thanks{Center for R{\&}D in Mathematics and Applications, Department of Mathematics, University of Aveiro, Portugal and Institute for Information Transmission Problems, Moscow, Russia, plakhov@ua.pt}}

\date{}

\begin{document}
\maketitle

\begin{abstract}
The aim of this paper is twofold. First, we cut off a part of a convex surface by a plane near a ridge point and characterize the limiting behavior of the surface measure in $S^2$ induced by this part of surface when the plane approaches the point.
Second, this characterization is applied to Newton's least resistance problem for convex bodies: minimize the functional $\int\!\!\int_\Om (1 + |\nabla u(x,y)|^2)^{-1} dx\, dy$ in the class of concave functions $u: \Om \to [0,M]$, where $\Om \subset \RRR^2$ is a convex body and $M > 0$. It has been known \cite{BFK} that if $u_*$ solves the problem then $|\nabla u_*(x,y)| \ge 1$ at all regular points $(x,y)$ such that $u_*(x,y) < M$. We prove that if the upper level set $L_M = \{ (x,y): u_*(x,y) = M \}$ has nonempty interior, then for almost all points of its boundary $(\bar x, \bar y) \in \pl L_M$ one has $\lim_{\stackrel{(x,y)\to(\bar x,\bar y)}{u_*(x,y)<M}}|\nabla u_*(x,y)| = 1$.
\end{abstract}

\begin{quote}
\end{quote}

\begin{quote}
\end{quote}

\section{Introduction}

\subsection{Local behavior of convex surfaces}

Consider a convex body\footnote{A convex body is a convex compact set with nonempty interior.} $C$ in the 3-dimensional space $\RRR^3$  with the coordinate $\rrr = (x,y,z)$, and a singular point $\rrr_0$ on its boundary, $\rrr_0 \in \pl C$. Let $\Pi$ be a plane of support to $C$ at $\rrr_0$. Consider the part of the surface $\pl C$ containing $\rrr_0$ cut off by a plane parallel to $\Pi$. We are interested in studying the limiting properties of this part of surface when the cutting plane approaches $\Pi.$

A singular point of the boundary, $\rrr_0$, is called a {\it conical point} if the tangent cone to $C$ at $\rrr_0$ is not degenerate, and a {\it ridge point} if the tangent cone degenerates into a dihedral angle (see, e.g., \cite{Pogorelov}). In this paper we consider ridge points, postponing the study of conical points to the future.

More precisely, let the tangent cone at $\rrr_0$ be given by
\beq\label{dih angle}
(\rrr - \rrr_0,\, \eee_1) \le 0, \qquad (\rrr - \rrr_0,\, \eee_2) \le 0,
\eeq
where $\eee_1$ and $\eee_2$ are non-collinear unit vectors, $\eee_1 \ne \pm\eee_2$. (Here and in what follows, $(\cdot\,, \cdot)$ means the scalar product.) Notice that the outward normals of all planes of support at $\rrr_0$ form the curve
$$
\Gam = \Gam_{\eee_1,\eee_2} = \{ v = \mu_1 \eee_1 + \mu_2 \eee_2 : |v| = 1,\, \mu_1 \ge 0,\, \mu_2 \ge 0 \} \subset S^2;
$$
it is the smaller arc of the great circle on $S^2$ through $\eee_1$ and $\eee_2$.

Let $\eee$ be a positive linear combination of $\eee_1$ and $\eee_2$, $\eee = \lam_1 \eee_1 + \lam_2 \eee_2$, $\lam_1 > 0$, $\lam_2 > 0$. For $t > 0$ consider the convex body
$$
C_t = C \cap \{ \rrr : (\rrr - \rrr_0,\, \eee) \ge -t \};
$$
it is the piece of $C$ cut off by the plane with the normal vector $e$ at the distance $t$ from $\rrr_0$. The body $C_t$ is bounded by the planar domain
$$
B_t = C \cap \{ \rrr : (\rrr - \rrr_0,\, \eee) = -t \}
$$
and the convex surface
\beq\label{St}
S_t = \pl C \cap \{ \rrr : (\rrr - \rrr_0,\, \eee) \ge -t \};
\eeq
that is, $\pl C_t = B_t \cup S_t$.

In what follows we denote by $|A|$ the standard Lebesgue measure (area) of the Borel set $A$ on the plane or on the convex surface $\pl C$. In particular, $|\square ABCD|$ means the area of the quadrangle $ABCD$. The same notation will be used for the length of a line segment or a curve; for instance, $|MN|$ means the length of the segment $MN$.

We denote by $\Pi^t$ the plane of equation $(\rrr - \rrr_0,\, \eee) = -t$ and by $\Pi_i$ $(i=1,\,2)$ the plane of equation $(\rrr - \rrr_0,\, \eee_i) = 0$.

Let $\pl' C$ be the set of regular points of $\pl C$; it is a full-measure subset of $\pl C$. Denote by $\nnn_\rrr$ the outward normal to $C$ at the point $\rrr \in \pl' C$. The surface measure of the convex body $C$ is the Borel measure $\nu_C$ in $S^2$ defined by
$$
\nu_C(A) := |\{ \rrr \in \pl' C : \nnn_\rrr \in A \}|
$$
for any Borel set $A \subset S^2$. It is well known that the surface measure satisfies the equation
\beq\label{center}
\int_{S^2} n\, d\nu_C(n) = \vec 0.
\eeq

Next we introduce the normalized measure $\nu_t = \nu_{t,\eee,\rrr_0,C}$ induced by the surface $S_t$. Namely, for any Borel set $A \subset S^2$ by definition we have
$$
\nu_t(A) := \frac{1}{|B_t|} |\{ \rrr \in S_t \cap \pl' C : \nnn_\rrr \in A \}|.
$$

The surface measure of the convex body $C_t$ equals $\nu_{C_t} = |B_t| \del_{-\eee} + |B_t| \nu_t$, hence $\int_{S^2} n\, d\nu_{C_t}(n) = |B_t| (-\eee + \int_{S^2} n\, d\nu_t(n)).$ Formula \eqref{center} applied to $C_t$ results in
\beq\label{centt}
\int_{S^2} n\, d\nu_t(n) = \eee.
\eeq

We say that $\nu_t$ weakly converges to $\nu_*$ as $t \to 0$, and use the notation $\nu_t \xrightarrow[t\to0]{} \nu_*$, if for any continuous function $f$ on $S^2$, $\lim_{t\to0} \int_{S^2} f(n)\, d\nu_t(n) = \int_{S^2} f(n)\, d\nu_*(n)$. Similarly, $\nu_*$ is called a partial weak limit of $\nu_t$, if there exists a sequence $t_i,\, i \in \mathbb{N}$ converging to zero such that for any continuous function $f$ on $S^2$, $\lim_{i\to\infty} \int_{S^2} f(n)\, d\nu_{t_i}(n) = \int_{S^2} f(n)\, d\nu_*(n)$.

We are interested in studying the properties of the weak limit (or a partial weak limit) $\nu_*$.

One of the properties is immediate: going to the limit $t \to 0$ or $t_i \to 0$ in formula \eqref{centt}, one obtains
\beq\label{cent}
\int_{S^2} n\, d\nu_*(n) = \eee.
\eeq

\begin{zam}\label{rem0} 
In the 2D case the following asymptotic property of a convex curve near a singular point is almost obvious.

Let $C \subset \RRR^2$ be a planar convex body and $\rrr_0$ be a singular point on its boundary. Let the tangent angle at $\rrr_0$ be given by \eqref{dih angle}, where $\eee_1$ and $\eee_2$ are non-collinear unit vectors, and let $\eee = \lam_1 \eee_1 + \lam_2 \eee_2$, $\lam_1 > 0$, $\lam_2 > 0$, $|\eee| = 1.$ Then the curve $S_t$ defined by \eqref{St} is the disjoint union of two curves, $S_t = S_t^1 \cup S_t^2$, where for $i = 1,\, 2$ the maximal deviation between normal vectors at $S_t^i$ and $\eee_i$ tends to zero, $\sup_{\xi\in S_t^i} (\nnn_\xi, \eee_i) \to 1$ as $t \to 0.$ Additionally, the lengths of $S_t^1$, $S_t^2$, and $B_t$ obey the following asymptotic relation (the sine law): there exist and coincide the following limits
$$
0 < \lim_{t\to0} \frac{|S_t^1|}{t\sin\bt} = \lim_{t\to0} \frac{|S_t^2|}{t\sin\al} = \lim_{t\to0} \frac{|B_t|}{t\sin(\pi-\al-\bt)} < +\infty,
$$
where $\al = \measuredangle \langle \eee_1, \eee \rangle \in (0,\, \pi)$, $\bt = \measuredangle \langle \eee_2, \eee \rangle \in (0,\, \pi)$, and therefore, $\sin\al = \sqrt{1 - (\eee_1, \eee)^2}$, $\sin\bt = \sqrt{1 - (\eee_2, \eee)^2}$, $\sin(\pi-\al-\bt) = (\eee_1, \eee)\sqrt{1 - (\eee_2, \eee)^2} + (\eee_2, \eee)\sqrt{1 - (\eee_1, \eee)^2}$
(see Fig.~\ref{fig0}). The proof of this formula is left to the reader.
     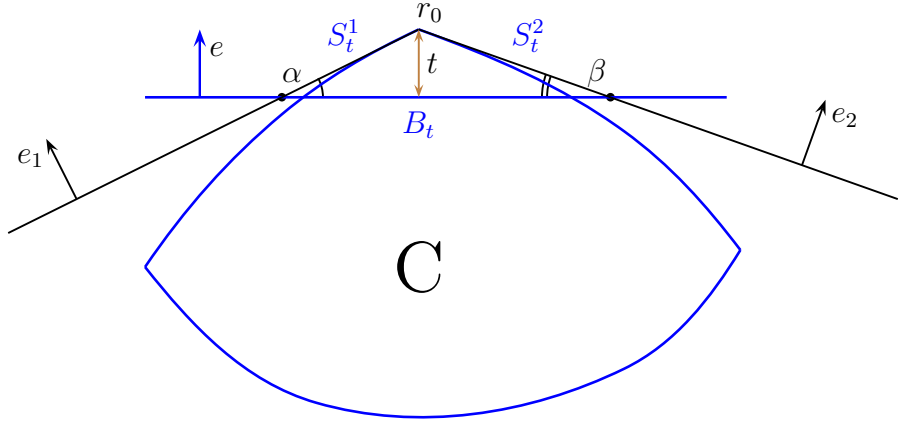
\begin{figure}[h]
\begin{picture}(0,160)
\scalebox{0.9}{
\rput(4,2.3){
\psdots(2,2.5)(6.8,2.5)
\psarc(2,2.5){0.6}{0}{28}
\rput(2.15,2.83){\scalebox{1.1}{$\alpha$}}
\psarc(6.8,2.5){1}{160}{180}
\psarc(6.8,2.5){0.925}{160}{180}
\rput(6.6,2.86){\scalebox{1.1}{$\bt$}}
\psecurve[linecolor=blue,linewidth=1.1pt](-1,-2)(0,0)(2,2.25)(4,3.5)(7,4.5)
\psecurve[linecolor=blue,linewidth=1.1pt](2,4)(4,3.5)(7,2)(8.7,0.25)(10,-2)
\psecurve[linecolor=blue,linewidth=1.1pt](-0.7,1.5)(0,0)(2.5,-2)(7,-1.5)(8.7,0.25)(9.3,2)
\psline[linecolor=blue,linewidth=1.1pt](0,2.5)(8.5,2.5)
\psline(-2,0.5)(4,3.5)
\psline(4,3.5)(11,1)
\psline[arrows=->,arrowscale=1.5](-1,1)(-1.444,1.888)
\rput(4.16,3.75){\scalebox{1.1}{$\rrr_0$}}
\rput(10.23,2.2){\scalebox{1.1}{$\eee_2$}}
\psline[arrows=->,arrowscale=1.5](9.6,1.5)(9.944,2.464)
\rput(-1.67,1.6){\scalebox{1.1}{$\eee_1$}}
\psline[arrows=->,arrowscale=1.5,linecolor=blue,linewidth=1pt](0.8,2.5)(0.8,3.5)
\rput(1.05,3.2){\scalebox{1.1}{$\eee$}}
 \rput(4,0){$\scalebox{2.5}{C}$}
\psline[linecolor=brown,arrows=<->,arrowscale=1.25](4,2.5)(4,3.5)
\rput(4.2,3){\scalebox{1.1}{$t$}}
\rput(2.9,3.4){\scalebox{1.1}{$\blue S_t^1$}}
\rput(5.6,3.4){\scalebox{1.1}{$\blue S_t^2$}}
\rput(4,2.1){\scalebox{1.1}{$\blue B_t$}}

}}
\end{picture}
\caption{The behavior of a planar convex curve near a singular point.}
\label{fig0}
\end{figure} 

The measure $\nu_t$ is defined in the exactly the same way as in the 3D case: for any Borel set $A \subset S^1$, $\nu_t(A) := \frac{1}{|B_t|} |\{ \rrr \in S_t \cap \pl' C : \nnn_\rrr \in A \}|.$ We see that in the 2D case $\nu_t$ weakly converges to the measure $\nu_*$ which is the sum of two atoms,
$$
\nu_* = \frac{\sin\bt}{\sin(\al+\bt)} \del_{\eee_1} + \frac{\sin\al}{\sin(\al+\bt)} \del_{\eee_2}.
$$
Using formula \eqref{cent}, one finds that $\frac{\sin\bt}{\sin(\al+\bt)} = \lam_1$ and $\frac{\sin\al}{\sin(\al+\bt)} = \lam_2$.
  \end{zam}

Consider two simple examples.

\begin{primer}\label{pr1}
Let $C$ be a convex polyhedron and $\rrr_0$ be an interior point of an edge of $C$. That is, $\rrr_0$ is a ridge point (see Fig.~\ref{fig00}).

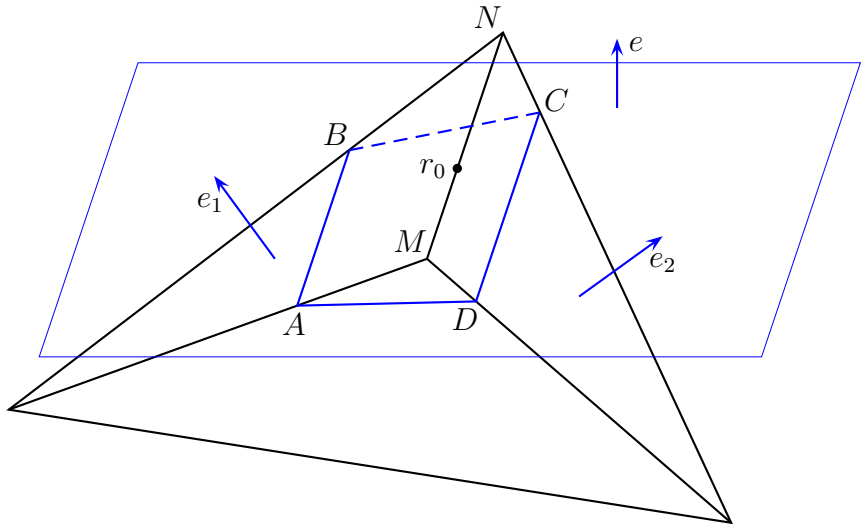
\begin{figure}[h]
\begin{picture}(0,200)
\scalebox{1}{
\rput(2.8,1.5){

\pspolygon(0,0)(6.5,5)(9.5,-1.5)
\psline(0,0)(5.5,2)(6.5,5)
\psline(5.5,2)(9.5,-1.5)
\psline[linecolor=blue](6.98,3.94)(6.145,1.435)(3.793,1.379)(4.48,3.44)
\psline[linecolor=blue,linestyle=dashed](4.48,3.44)(6.98,3.94)
\pspolygon[linecolor=blue,linewidth=0.3pt](0.4,0.7)(9.9,0.7)(11.2,4.6)(1.7,4.6)
\psdots[dotsize=3.5pt](5.9,3.2)
\rput(5.58,3.2){$\rrr_0$}
\psline[linecolor=blue,arrows=->,arrowscale=1.5,linewidth=0.8pt](7.5,1.5)(8.6,2.3)
\psline[linecolor=blue,arrows=->,arrowscale=1.5,linewidth=0.8pt](3.5,2)(2.7,3.1)
\psline[linecolor=blue,arrows=->,arrowscale=1.5,linewidth=0.8pt](8,4)(8,4.92)
\rput(6.3,5.2){$N$}
\rput(5.28,2.23){$M$}
\rput(3.75,1.13){$A$}
\rput(4.3,3.65){$B$}
\rput(7.2,4.1){$C$}
\rput(6,1.2){$D$}
\rput(2.65,2.75){$\eee_1$}
\rput(8.6,1.95){$\eee_2$}
\rput(8.25,4.85){$\eee$}
}}
\end{picture}
\caption{$C$ is the tetrahedron, $\rrr_0$ is situated on its edge $MN$, and the section of $C$ by the plane $(\rrr - \rrr_0,\, \eee) = -t$ is the quadrilateral $ABCD$.}
\label{fig00}
\end{figure} 

The surface $S_t$ is composed of two quadrilaterals $MNBA$ and $MNCD$ and two triangles $BCN$ and $ADM$. The outward normals to the quadrilaterals $MNBA$ and $MNCD$ are $\eee_1$ and $\eee_2$, respectively, and their areas are of the order of $t$, $|\square MNBA| = c_1 t + O(t^2)$, $|\Box MNCD| = c_2 t + O(t^2)$, $c_1 > 0$, $c_2 > 0$. The areas of the triangles $BCN$ and $ADM$ are $O(t^2)$.

The planar surface $B_t$ is the quadrilateral $ABCD$, the normal to this surface is $\eee = \lam_1 \eee_1 + \lam_2 \eee_2$, and its area is of the order of $t$, $|\square ABCD| = c_0 t + O(t^2)$. It follows that the corresponding measure $\nu_t$ converges to the measure $\nu_*$ supported on the two-point set $\{ e_1,\, e_2 \}$, $\nu_* = \frac{c_1}{c_0} \del_{e_1} + \frac{c_2}{c_0} \del_{e_2}$. Using formula \eqref{cent}, one obtains that $\frac{c_1}{c_0} = \lam_1$, $\frac{c_2}{c_0} = \lam_2$,
and therefore,
\beq\label{limit12}
\nu_* = \lam_1 \del_{\eee_1} + \lam_2 \del_{\eee_2}.
\eeq
\end{primer}

\begin{primer}\label{pr2}
Let $C = \{ \rrr = (x,y,z) : x^2 + y^2 \le 1,\, 0 \le z \le 1 \}$ be a cylinder, and take the ridge point $\rrr_0 = (1,0,0) \in \pl C$. The tangent cone at $\rrr_0$ is a dihedral angle with the outward normals $\eee_1 = (1,0,0)$ and $\eee_2 = (0,0,-1)$. Take a unit vector $\eee = \lam_1 \eee_1 + \lam_2 \eee_2$ with $\lam_1 > 0$, $\lam_2 > 0$, $\lam_1^2 +  \lam_2^2 = 1$. See Fig.~\ref{fig01}.

     \begin{figure}[h]
\begin{picture}(0,200)
\scalebox{1}{
\rput(2.8,1.8){

\psline(0,0)(0,4)
\psellipse(3,0)(3,0.75)
\pspolygon[fillstyle=solid,fillcolor=white,linewidth=0pt,linecolor=white](0.05,0)(5.96,0)(5.96,1)(0.05,1)
\psellipse[linewidth=0.5pt,linestyle=dashed](3,0)(3,0.75)
\psellipse(3,4)(3,0.75)
\psdots[dotsize=3.5pt](6,0)
\rput(6.3,0){$\rrr_0$}
\psline[linecolor=blue,linestyle=dashed,linewidth=0.5pt](4.62,-0.64)(4.96,0.53)
\psecurve[linewidth=0.8pt,linecolor=blue]
(4.3,-0.96)(4.62,-0.64)(4.9,-0.36)(5.2,-0.03)(5.5,0.31)(5.8,0.68)(5.91,0.84)(6.01,1.12)(5.9,1.13)
\psecurve[linewidth=0.5pt,linestyle=dashed,linecolor=blue]
(5.91,0.84)(6.01,1.12)(5.9,1.13)(5.6,0.98)(5.3,0.785)(4.96,0.53)(4.7,0.335)
\psline(6,0)(6,4)
\psline[linewidth=0.3pt,linecolor=blue](4.45,-0.67)(4.12,-1.8)(7.12,0.7)(8,3.7)(6,2.03)
\psline[linestyle=dashed,linewidth=0.3pt,linecolor=blue](6,2.03)(5,1.2)(4.45,-0.67)
\psline[arrows=->,arrowscale=1.5,linewidth=0.8pt](6,3.1)(7.1,3.1)
\rput(6.9,3.4){$\eee_1$}
\psline[arrows=->,arrowscale=1.5,linewidth=0.8pt](3.4,-0.7)(3.4,-1.8)
\rput(3.1,-1.6){$\eee_2$}
\psline[arrows=->,arrowscale=1.5,linewidth=0.8pt](7.2,1)(7.85,0.5)
\rput(7.9,0.8){$\eee$}
}}
\end{picture}
\caption{$C$ is the cylinder and $\rrr_0$ is located on the boundary of its rear disc.}
\label{fig01}
\end{figure}
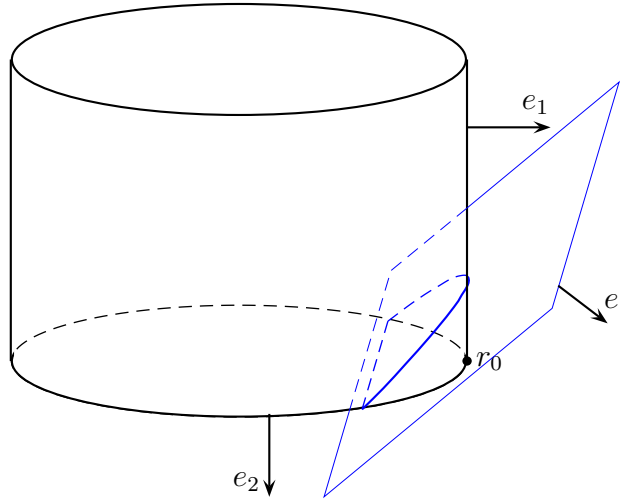 

The surface $S_t$ is the union of the piece of the cylindrical surface, $S_t^1$, and the piece of the rear disc, $S_t^2$, cut off by the plane $(\rrr - \rrr_0,\, \eee) = -t$. The outward normals at points of $S_t^2$ coincide with $\eee_2$. The outward normals at points of $S_t^1$ are contained in a neighborhood of $\eee_1$ that shrinks to $\eee_1$ when $t \to 0$. Hence $\nu_t$ is the sum of two terms, where the former one is proportional to $\del_{\eee_2}$ and the latter one is proportional to a measure that weakly converges to  $\del_{\eee_2}$. Using formula \eqref{cent}, one concludes that $\nu_t$ converges to the measure $\nu_*$ given by \eqref{limit12}.
\end{primer}

It may seem that the limiting measure is always given by \eqref{limit12}, as in the 2D case and in examples \ref{pr1} and \ref{pr2}, but this is not the case. Consider two more examples.

\begin{primer}\label{pr3}
Let $C$ be the part of a cylinder bounded by two planes, $C = \{ \rrr = (x,y,z) : -z-1 \le x \le z+1,\, y^2 + z^2 \le 1 \}$, and take the ridge point $\rrr_0 = (0,0,-1) \in \pl C$. The outward vectors of the corresponding dihedral angle are $\eee_1 = \frac{1}{\sqrt{2}} (-1,0,-1)$ and $\eee_2 = \frac{1}{\sqrt{2}} (1,0,-1)$. We take $\eee = (0,0,-1) = \frac{1}{\sqrt{2}} \eee_1 + \frac{1}{\sqrt{2}} \eee_2$. See Fig.~\ref{fig02}.

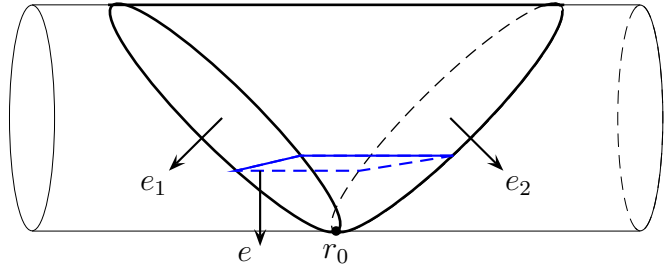
\begin{figure}[h]
\begin{picture}(0,120)
\scalebox{1}{
\rput(7.5,0.75){
\psline[linewidth=0.2pt](-4,0)(4,0)
\psline[linewidth=0.2pt](-4,3)(4,3)
\psellipse[linewidth=0.2pt](-4,1.5)(0.3,1.5)
\psellipse[linewidth=0.2pt](4,1.5)(0.3,1.5)
\rput(0,-0.3){$\rrr_0$}
\pspolygon[fillstyle=solid,fillcolor=white,linewidth=0pt,linecolor=white](3.5,0.05)(4,0.05)(4,2.95)(3.5,2.95)
\psellipse[linewidth=0.2pt,linestyle=dashed](4,1.5)(0.3,1.5)
\rput{-45}(0.4,0.44){\psellipse[linewidth=1pt](0,1.5)(0.4242,2.121)}
\pspolygon[fillstyle=solid,fillcolor=white,linewidth=0pt,linecolor=white](-1.5,0.03)(0,0.03)(3,2.97)(1.5,2.97)
\rput{-45}(0.4,0.44){\psellipse[linewidth=0.5pt,linestyle=dashed](0,1.5)(0.4242,2.121)}
\rput{45}(-0.4,0.44){\psellipse[linewidth=1pt](0,1.5)(0.4242,2.121)}
\psline[linewidth=1pt](-3,3)(3,3)
\psdots(0,0)
\psline[linewidth=0.8pt,linecolor=blue](1.54,1)(-0.47,1)(-1.32,0.8)
\psline[linewidth=0.8pt,linecolor=blue,linestyle=dashed](1.54,1)(-0.47,1)(-1.32,0.8)(0.31,0.8)(1.54,1)
\psline[arrows=->,arrowscale=1.5,linewidth=0.8pt](-1.5,1.5)(-2.2,0.8)
\rput(-2.4,0.6){$\eee_1$}
\psline[arrows=->,arrowscale=1.5,linewidth=0.8pt](1.5,1.5)(2.2,0.8)
\rput(2.4,0.6){$\eee_2$}
\psline[arrows=->,arrowscale=1.5,linewidth=0.8pt](-1,0.8)(-1,-0.2)
\rput(-1.2,-0.3){$\eee$}
}}
\end{picture}
\caption{$C$ is the part of a cylinder bounded by two planes through $\rrr_0$.}
\label{fig02}
\end{figure} 

We have $C_t = C \cap \{ z \le -1 + t \}$, and $B_t$ is the rectangle $-t \le x \le t,\ -\sqrt{2t - t^2} \le y \le \sqrt{2t - t^2}$ in the plane $z = -1 + t$. The surface $S_t$ is the union of three parts, $S_t = S_t^1 \cup S_t^2 \cup S_t^0$, where $S_t^1$ is the planar domain of equations $x = -z - 1,\, x \ge -t,\, (x+1)^2 + y^2 \le 1$ with the outward normal $\eee_1$ and $S_t^2$ is the planar domain of equations $x = z + 1,\, x \le t,\, (x-1)^2 + y^2 \le 1$ with the outward normal $\eee_2$.  $S_t^1$ and $S_t^2$ are segments of ellipses in the planes $x = -z - 1$ and $x = z + 1$, respectively. The surface $S_t^0$ is the graph of the function $z(x,y) = -\sqrt{1 - y^2}$ defined on the domain $-(1 - \sqrt{1-y^2}) \le x \le 1 - \sqrt{1-y^2}$, $-\sqrt{2t - t^2} \le y \le \sqrt{2t - t^2}$; see Fig.~\ref{fig03}. The outward normals to $S_t^0$ are contained in a neighborhood of $\eee$ shrinking to $\eee$ when $t \to 0$.

      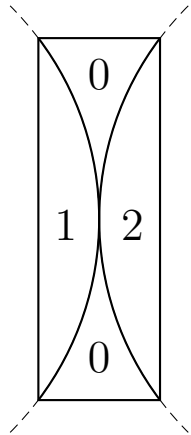
\begin{figure}[h]
\begin{picture}(0,160)
\scalebox{0.8}{
\rput(9,3.5){
\psarc[linewidth=0.5pt,linestyle=dashed](-5,0){5}{-45}{45}
\psarc[linewidth=0.5pt,linestyle=dashed](5,0){5}{135}{225}
\psarc[linewidth=1pt](-5,0){5}{-36.87}{36.87}
\psarc[linewidth=1pt](5,0){5}{143.13}{216.87}
\pspolygon[linewidth=1pt](1,-3)(1,3)(-1,3)(-1,-3)
\rput(-0.55,-0.1){\scalebox{1.8}{$1$}}
\rput(0.55,-0.1){\scalebox{1.8}{$2$}}
\rput(0,-2.3){\scalebox{1.8}{$0$}}
\rput(0,2.4){\scalebox{1.8}{$0$}}
}}
\end{picture}
\caption{The projections of $S_t^1$, $S_t^2$, and $S_t^0$ on the $xy$-plane, marked by $"1",\, "2"$, and $"0"$, respectively.}
\label{fig03}
\end{figure} 

The areas of the surfaces are easy to calculate,
$$
|B_t| = 2t\cdot 2\sqrt{2t - t^2} = 4\sqrt 2\, t^{3/2}(1 + o(1)),\ t \to 0,
$$
$$
|S_t^0| = \frac{4\sqrt 2}{3}\, t^{3/2}(1 + o(1)), \quad \text{and} \quad |S_t^1| = |S_t^2| = \frac{8}{3}\, t^{3/2}(1 + o(1)),\ t \to 0.
$$
It follows that $\nu_t$ converges to the measure
$$\nu_* = \frac{\sqrt{2}}{3} \del_{\eee_1} + \frac{\sqrt{2}}{3} \del_{\eee_2} + \frac{1}{3} \del_{\eee}$$
supported on the three-point set $\{ \eee_1,\, \eee_2,\, \eee \}$.
\end{primer}


Thus, the limiting measure is not necessarily the sum of two atoms in the form \eqref{limit12}. Moreover, $\nu_t$ may not converge, as shown in the following example.

\begin{primer}\label{pr4}
Let $C$ be defined by $C = \{ \rrr = (x,y,z) : -z \le x \le z,\ \gam(y) \le z \le 1 \}$, where $\gam : \RRR \to \RRR$ is a convex even function such that $\gam(0) = 0$ and $\gam(y) > 0$ for $y \ne 0$. The point $\rrr_0 = (0,0,0)$ is a ridge point, the corresponding tangent cone is given by \eqref{dih angle} with $\eee_1 = \frac{1}{\sqrt{2}} (-1,0,-1)$ and $\eee_2 = \frac{1}{\sqrt{2}} (1,0,-1)$, and we take $\eee = (0,0,-1) = \frac{1}{\sqrt{2}} \eee_1 + \frac{1}{\sqrt{2}} \eee_2$. That is, $\eee_1,\, \eee_2,$ and $\eee$ are as in the previous example.

One has $C_t = C \cap \{ z \le t \}$, and for $t \le 1$, $B_t$ is the rectangle $-t \le x \le t,\ \gam(y) \le t$ in the plane $z = t$. Again we have $S_t = S_t^1 \cup S_t^2 \cup S_t^0$, where $S_t^1$ is the planar domain of equations $x = -z$, $-t \le x \le -\gam(y)$ with the outward normal $\eee_1$, $S_t^2$ is the planar domain of equations $x = z$, $\gam(y) \le x \le t$ with the outward normal $\eee_2$, and $S_t^0$ is the graph of the function $z(x,y) = \gam(y)$ defined on the domain $-\gam(y) \le x \le \gam(y)$, $\gam(y) \le t$. The outward normals to $S_t^0$ are contained in a neighborhood of $\eee$ shrinking to $\eee$ when $t \to 0$.

The projections of $S_t^1$, $S_t^2$, and $S_t^0$ on the $xy$-plane look like those shown in Fig.~\ref{fig03}. If the family $\big( |S_t^-|/|B_t|,\, |S_t^+|/|B_t|,\, |S_t^0|/|B_t| \big)$ has a partial limit equal to $(b^-, b^+, b^0)$ as $t \to 0$, then the family of measures $\nu_t$ has the partial limit $\nu_* = b^- \del_{\eee_1} + b^+ \del_{\eee_2} + b^0 \del_{\eee}$.

Let the graph $x = \gam(y)$, $y > 0$ be a broken line with infinitely many segments and with the vertices $(x_i, y_i)$, $i \ge i_0$ that are defined inductively. The initial values $x_{i_0} > 0$, $y_{i_0} > 0$ are arbitrary. Take $0 < a < b < 1$. Given $(x_i, y_i)$, put $x_{i+1} = x_i/i$; if $i$ is even then put $y_{i+1} = ay_i$, and if $i$ is odd then put $y_{i+1} = by_i$. The initial value $i_0$ is taken sufficiently large, so that the corresponding function $\gam$ is convex.

Taking $t = x_i$, we have $|B_{x_i}| = 4x_i y_i$; for $i$ even
$$
|S_{x_i}^0| = 2(1-a)\, x_i y_i (1 + o(1)) \ \ \, \text{and} \ \ \, |S_{x_i}^1| = |S_{x_i}^2| = \sqrt 2 (1+a)\, x_i y_i (1 + o(1)), \ \, i \to \infty,
$$
and  for $i$ odd
$$
|S_{x_i}^0| = 2(1-b)\, x_i y_i (1 + o(1)) \ \ \, \text{and} \ \ \, |S_{x_i}^1| = |S_{x_i}^2| = \sqrt 2 (1+b)\, x_i y_i (1 + o(1)), \ \, i \to \infty.
$$
It follows that there are at least two partial limits of $\nu_t$,
$$
\nu_*^1 = \lim_{\stackrel{i\text{ even}}{i\to\infty}} \nu_{x_i} = \frac{1 + a}{2\sqrt 2} \big( \del_{\eee_1} + \del_{\eee_2} \big) + \frac{1 - a}{2}\, \del_\eee
$$ and $$
 \nu_*^2 = \lim_{\stackrel{i\text{ odd}}{i\to\infty}} \nu_{x_i} = \frac{1 + b}{2\sqrt 2} \big( \del_{\eee_1} + \del_{\eee_2} \big) + \frac{1 - b}{2}\, \del_\eee.
$$
\end{primer}

The following Theorems \ref{t2 segment}, \ref{t1}, and \ref{t1b} describe the limiting behavior of $\nu_t$.

Note that the set $B_0 = C \cap \{ \rrr : (\rrr - \rrr_0,\, \eee) = 0 \}$ is the intersection of the body's surface and the edge of the dihedral angle. It is either a non-degenerate line segment containing $\rrr_0$, or the singleton $\{ \rrr_0 \}$. If it is a line segment then $\nu_t$ converges and the limit is the sum of two atoms given by \eqref{limit12}, as claims Theorem \ref{t2 segment}. In general (when this intersection may be both a line segment and the singleton) the limiting behavior of $\nu_t$ is more complicated and is described by Theorems \ref{t1} and \ref{t1b}.

\begin{theorem}\label{t2 segment}
If $B_0$ is a non-degenerate line segment then $\nu_t \xrightarrow[t \to 0]{} \lam_1 \del_{\eee_1} + \lam_2 \del_{\eee_2}$.
\end{theorem}

Theorem \ref{t1} states that the support of each partial limit of $\nu_t$ is contained in the arc $\Gam$ and contains its endpoints. Theorem \ref{t1b} states that, vice versa, each compact subset of an arc containing its endpoints can be realized as the support of the limit of a family of measures $\nu_t$.

\begin{theorem}\label{t1}
The set of partial limits of $\nu_t$ as $t \to 0$ is nonempty, and each partial limit  is supported on a closed subset of $\Gam$ containing $\eee_1$ and $\eee_2.$
\end{theorem}

\begin{theorem}\label{t1b}
Let $\eee_1$ and $\eee_2$ be two unit vectors, $\eee_1 \ne \pm \eee_2$, and let $\eee = \lam_1 \eee_1 +\lam_2 \eee_2$, $\lam_1 > 0$, $\lam_2 > 0$, $|\eee| = 1$. Assume that $K$ is a closed subset of the arc $\Gam = \Gam_{\eee_1,\eee_2}$ containing its endpoints $\eee_1$ and $\eee_2$,\, $\{ \eee_1,\, \eee_2 \} \subset K \subset \Gam$.  Then there exist a convex body $C$ and a ridge point $\rrr_0$ on its surface such that the tangent cone at $\rrr_0$ is given by the inequalities $(\rrr - \rrr_0,\, \eee_1) \le 0, \ (\rrr - \rrr_0,\, \eee_2) \le 0$ and the measure $\nu_t = \nu_{t,\eee, \rrr_0, C}$ weakly converges as $t \to 0$ to a measure $\nu_*$ such that spt$\,\nu_* = K$.
\end{theorem}

\subsection{Application to Newton's least resistance problem}

The main motivation for this study came from Newton's problem of minimal resistance.

The problem is as follows. Consider a convex body $C$ moving forward in a homogeneous medium composed of point particles. The medium is extremely rare, so as mutual interaction of particles is neglected. There is no thermal motion of particles, that is, the particles are initially at rest. When colliding with the body, each particle is reflected elastically. As a result of collisions, there appears the drag force that acts on the body and slows down its motion.

Take a coordinate system with the coordinates $x,y,z$ connected with the body such that the $z$-axis is parallel and co-directional to the velocity of the body. Let the upper part of the body's surface be the graph of the concave function $u = u_C : \Om \to \RRR$, where $\Om = \Om_C$ is the projection of $C$ on the $xy$-plane. Then the $z$-component of the drag force equals $-2\rho v^2 F(u)$, where $v$ is the velocity of the body, $\rho$ is the density of the medium, and
\beq\label{resN}
F(u) = \int\!\!\!\int_\Om \frac{1}{1 + |\nabla u(x,y)|^2}\, dx dy.
\eeq
$F(u)$ is called the resistance of the body. See Fig.~\ref{figRes}.

       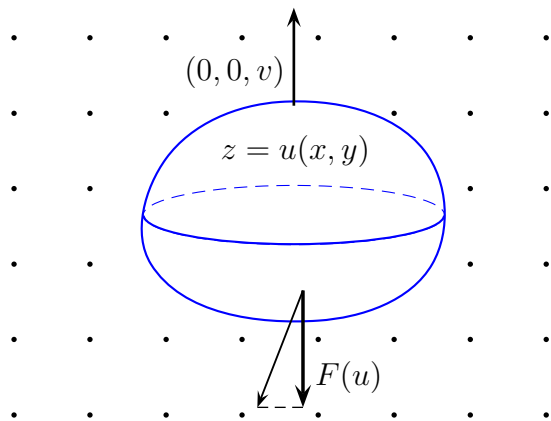
\begin{figure}[h]
\begin{picture}(0,165)
\scalebox{1}{
\rput(6.25,1.2){

\psellipse[linecolor=blue,linewidth=0.8pt](1.68,1.65)(1.99,0.4)
\pspolygon[fillstyle=solid,fillcolor=white,linewidth=0pt,linecolor=white](-0.4,1.65)(4,1.65)(4,2.5)(-0.4,2.5)
\psecurve[linecolor=blue,linewidth=0.8pt](-0.2,1)(2,0.25)(3.5,1)(3.4,2.5)(1.5,3.15)(0,2.4)(-0.2,1)(2,0.25)(3.5,1)
\psellipse[linecolor=blue,linewidth=0.4pt,linestyle=dashed](1.68,1.65)(1.99,0.4)
\psline[arrows=->,arrowscale=1.5,linewidth=1.3pt](1.8,0.65)(1.8,-0.9)
\psline[arrows=->,arrowscale=1.5,linewidth=0.7pt](1.8,0.65)(1.2,-0.9)
\psline[linestyle=dashed,linewidth=0.5pt](1.2,-0.9)(1.8,-0.9)
\rput(2.4,-0.5){$F(u)$}
\psline[arrows=->,arrowscale=1.5,linewidth=1pt](1.68,3.1)(1.68,4.4)
\rput(0.9,3.55){$(0,0,v)$}
\rput(1.7,2.5){$z = u(x,y)$}
\psdots[dotsize=2pt](-2,-1)(-2,0)(-2,1)(-2,2)(-2,3)(-2,4)
(-1,-1)(-1,0)(-1,1)(-1,2)(-1,3)(-1,4)
(0,-1)(0,0)(0,3)(0,4)
(1,-1)(1,0)(1,4)
(2,-1)(2,0)(2,4)
(3,-1)(3,0)(3,3)(3,4)
(4,-1)(4,0)(4,1)(4,2)(4,3)(4,4)
(5,-1)(5,0)(5,1)(5,2)(5,3)(5,4)

}}
\end{picture}
\caption{A convex body moving in a rarefied medium.}
\label{figRes}
\end{figure}

The problem consists of minimizing the resistance in a certain class of bodies. Initially the problem was considered by I. Newton \cite{N} in the class of bodies symmetric with respect to the $z$-axis that have fixed projections on the $z$-axis and on the $xy$-plane (of course, the latter one is a circle). In Fig.~\ref{figNewtonOpt} there is shown the solution in the case when the length of the projection of the body on the $z$-axis is equal to the diameter of its projection on the $xy$-plane.

       \begin{figure}[h]
\centering
\includegraphics[scale=0.45]{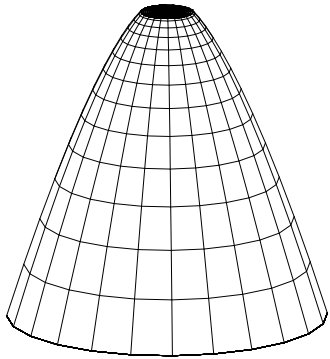}
\caption{A solution to the rotationally symmetric Newton problem.}
\label{figNewtonOpt}
\end{figure} 

The new interest to the problem was triggered in 1993 by the paper of Buttazzo and Kawohl \cite{BK}. Since then, the problem of minimal resistance was stated and studied in various classes of bodies, among them: convex rotationally symmetric bodies with the fixed arc length \cite{BelloniKawohl} and with the fixed volume \cite{BW prescribed volume}, convex bodies with developable lateral surface \cite{LP1}, bodies moving in a medium with thermal motion of particles \cite{PT thermal}. A generalized Newton problem, with the resistance being written down in the form of a surface integral, was considered in \cite{BG97}.

The problem was also generalized to various classes of nonconvex bodies. If each particle hits the body at most once (single impact condition, SIC) then the formula for resistance \eqref{resN} remains valid. The problem was studied in several classes of bodies satisfying SIC in the papers  \cite{BK,BFK,CL1,CL2,PlakhovSIA}.
If multiple particle-body collisions are allowed, there is no explicit analytic formula for the resistance (formula \eqref{resN} takes into account only the first reflections of particles with the body); however, the problem can be studied using the methods of billiards \cite{optimal roughening,bookP}.

The problems of minimal and maximal resistance were also studied in the case when the body, along with the translation, performs a rotational motion \cite{Nonlinearity,ARMA,rough2D,Magnus}. In this case an interesting analogy with the Magnus effect and with optical retroreflectors is found and discussed. Again, formula \eqref{resN} for the resistance is not valid here, but the problem can be studied using the methods of optimal mass transport.

Here we concentrate on the following problem: given a positive number $M$ and a planar convex body $\Om \subset \RRR^2$, minimize the functional \eqref{resN} in the class of concave functions $u : \Om \to [0,\, M]$. Equivalently, one can ask for minimizing the resistance in the class of convex bodies that have the projection $[0,\, M]$ on the $z$-axis and the projection $\Om$ on the $xy$-plane. The only difference as compared with the original Newton's problem is that the rotational symmetry is not required here.

It is the immediate and the earliest \cite{BK} generalization of the classical Newton's problem. However, despite the apparent simplicity of the statement of the problem, it remains open since 1993. Let us mention the known results. First, the solution $u_*$ to the problem exists \cite{BFK} and (if $\Om$ is a circle) does not coincide with the solution found by Newton in the rotationally symmetric case \cite{BrFK}. Second, if the solution $u_*$ is of class $C^2$ in an open subset of $\Om$, then det$\Big(\! \begin{array}{cc} u_{xx} & u_{xy}\\ u_{xy} & u_{yy} \end{array} \!\Big) = 0$ in this subset \cite{BrFK}. (Notice, however, that the existence of such an open set is not proved, so this statement may happen to describe a nonexistent object.) Third, if $u_*(x,y) < M$ and $\nabla u_*(x,y)$ exists, then $|\nabla u_*(x,y)| \ge 1$ \cite{BFK}.

A numerical study of this problem was made in \cite{LO} and \cite{W}  for the case when $\Om$ is a circle. According to the results of this study, the upper level set $L_M = \{ (x,y) : u_*(x,y) = M \}$ is either a line segment or a regular polygon, and the centers of $L_M$ and $\Om$ coincide. The set $L_M$ is a segment, if $M$ is greater than a certain value $M_c$ (according to \cite{W}, $M_c \approx 1.5$), and is a polygon otherwise, and the number of sides of the polygon is a piecewise constant monotone decreasing function of the parameter $M$  going to infinity as $M \to 0$. Besides, the set of singular points of $u_*$ seems to be the union of several (finitely many) curves in $\Om$.

The following fact was also established by numerical methods\footnote{G. Wachsmuth, personal communication.}: the supremum of the set of values $|\nabla u_*(x,y)|$ over all regular points $(x,y) \in \Om$ seems to be equal to 1, if $L_M$ nas nonempty interior, and smaller than 1, if $L_M$ is a segment. We prove the following theorem justifying the first part of this statement.

Let $M > 0$, and let $\Om \subset \RRR^2$ be a planar convex body.

\begin{theorem}\label{t3}
Let $u_*$ minimize the functional \eqref{resN} in the class of concave functions $u : \Om \to [0,\, M]$, and let the upper level set $L_M  = \{ (x,y) : u_*(x,y) = M \}$ have nonempty interior. Then for almost all points $(\bar x, \bar y) \in \pl L_M$,
$$
\lim_{\stackrel{(x,y) \to (\bar x, \bar y)}{(x,y) \in \Om \setminus L_M}} |\nabla u_*(x,y)| = 1.
$$
\end{theorem}

The proof of Theorem \ref{t3} is based on Theorem \ref{t1}.

\begin{zam}
It was conjectured in 1995 in \cite{BFK} (Remark 6.3) that the slope of the optimal surface is 1 along the boundary $\pl L_M$. Our Theorem \ref{t3} gives the affirmative answer to this conjecture in the case when $L_M$ has nonempty interior. On the other hand, the numerical evidence seems to suggest that the conjecture is not true in the case when $L_M$ is a segment.
\end{zam}

It is well known that $L_M$ is not empty and is not a singleton; therefore it is either a line segment, or a planar convex set with nonempty interior. It is also true that if $\Om$ contains a circle of radius greater than $M$ (for instance, if $\Om$ is a unit circle and $M < 1$) then $L_M$ has nonempty interior. I could not find the proof of this fact in the literature, therefore I provide the proof below.

\begin{utv}\label{utv vnutr}
Let $M_0 = M_0(\Om)$ be the maximal radius of a circle that can be put inside $\Om$. Then for $M < M_0$ the set $L_M = \{ (x,y) : u_*(x,y) = M \}$ has nonempty interior. In particular, if $\Om$ is a unit circle and $M < 1$ then $L_M$ has nonempty interior. Thereby Theorem \ref{t3} is applicable for these values of $M$.
\end{utv}

\begin{proof}
Consider the convex body $C = C_{(u_*)} = \{ (x,y,z) \in \RRR^3 : (x,y) \in \Om,\ 0 \le z \le u_*(x,y) \}$. Note that for any point $(x,y)$ in $\overline{\Om \setminus L_M}$ there is a plane of support through $(x,y)$ with the slope greater than or equal to 1.

Indeed, take a sequence of regular points $(x_i,y_i)$ in $\Om \setminus L_M$ converging to $(x,y)$; the tangent planes at these point have the slope at least 1. The sequence of normals to these planes has a limiting point, say $v$. Then the plane through $(x,y)$ orthogonal to $v$ is a plane of support, and its slope is greater than or equal to 1.

Take a point $(x,y)$ on $\pl L_M$ and draw a plane of support with the slope $k \ge 1$ through this point. Consider the lines of intersection of this plane with the horizontal planes $z = M$ and $z = 0$. The projections of these planes on the $xy$-plane are again parallel lines, and the distance between them is $M/k \le M$ (see Fig.~\ref{figL}).

       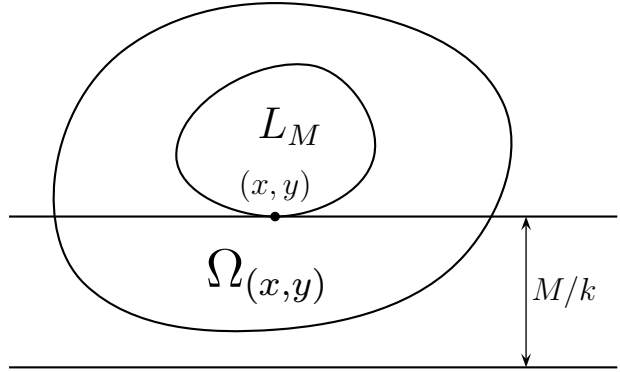
\begin{figure}[h]
\begin{picture}(0,150)
\scalebox{1}{
\rput(4,0.2){
\psline(0,0)(8,0)
\psline(0,2)(8,2)
\psecurve(0.6,2.5)(1,1.1)(3.5,0.5)(6,1.5)(6.5,3.5)(4,4.8)(1.5,4.2)(0.6,2.5)(1,1.1)(3.5,0.5)
\psecurve(2.2,2.8)(3.5,2)(4.8,2.8)(4,4)(2.2,2.8)(3.5,2)(4.8,2.8)
\psdots(3.5,2)
\psline[arrows=<->,arrowscale=1.5,linewidth=0.5pt](6.8,0)(6.8,2)
\rput(7.3,1){$M/k$}
\rput(3.5,2.4){$(x,y)$}
\rput(3.4,1.2){\scalebox{1.7}{$\Om_{(x,y)}$}}
\rput(3.7,3.2){\scalebox{1.4}{$L_M$}}
}}
\end{picture}
\caption{The auxiliary construction in Proposition \ref{utv vnutr}.}
\label{figL}
\end{figure} 

Denote by $\Om_{(x,y)}$ the part of $\Om$ between (or on) these lines. The latter line (corresponding to the intersection with the plane $z=0$) lies outside $\Om$, and therefore the distance between any point of $\Om_{(x,y)}$ and $\pl\Om$ does not exceed $M$.

Let $D_{M_0}$ be a ball with radius $M_0$ contained in $\Om$, and denote by $D_c$ the concentric open ball with the radius $c$. Let $M < M_0$; then $\Om_{(x,y)}$ does not intersect $D_{M_0-M}$. The set $L_M$ can be represented as $L_M = \overline{\Om \setminus \big( \cup_{(x,y)\in\pl L_M}  \Om_{(x,y)} \big) }$, and therefore, contains $D_{M_0-M}$.
\end{proof}

Theorems \ref{t2 segment}, \ref{t1}, \ref{t1b}, and \ref{t3} are proved in Sections \ref{sect t2}, \ref{sect t1}, \ref{sect t1b}, and \ref{sect t3}, respectively.

\section{Proof of Theorem \ref{t2 segment}}\label{sect t2}

First we are going to estimate the area of $B_t$.

Recall that the planes $\Pi_i :\, (\rrr - \rrr_0,\, \eee_i) = 0,\,  i = 1,\,2$ are the faces of the dihedral angle and $\Pi^t : \, (\rrr - \rrr_0,\, \eee) = -t$ is the crossing plane.

Let the length of the segment $B_0$ be $|B_0| = l > 0$. Denote by $M$ and $N$ its endpoints. Draw the planes through $M$ and $N$ orthogonal to $B_0$ and denote by $M_{i}^t$ and $N_{i}^t$, $i = 1,\,2$, the points of intersection of these planes with the lines $l_i^t = \Pi_i \cap \Pi^t$; see Fig.~\ref{fig4}\,(a).

        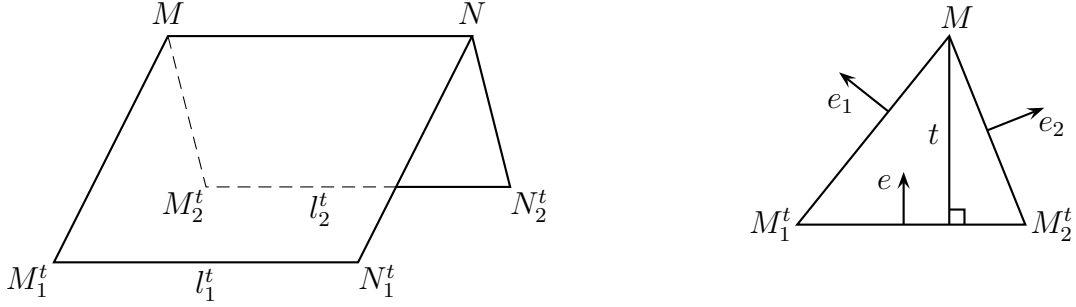
\begin{figure}[h]
\begin{picture}(0,105)
\scalebox{1}{
\rput(1,0.25){
\pspolygon(0,0)(4,0)(5.5,3)(1.5,3)
\psline(5.5,3)(6,1)(4.5,1)
\psline[linestyle=dashed,linewidth=0.4pt](1.5,3)(2,1)(4.5,1)
\rput(1.5,3.3){$M$}
\rput(5.5,3.3){$N$}
\rput(-0.35,-0.25){$M_1^t$}
\rput(4.25,-0.25){$N_1^t$}
  \rput(2,-0.3){$l_1^t$}
\rput(1.7,0.75){$M_2^t$}
\rput(6.25,0.75){$N_2^t$}
   \rput(3.5,0.7){$l_2^t$}
}}
\scalebox{1}{
\rput(10.5,0.75){
\pspolygon(0,0)(3,0)(2,2.5)
\psline[arrows=->,arrowscale=1.5,linewidth=0.8pt](1.4,0)(1.4,0.7)
\rput(1.15,0.56){$\eee$}
\psline[arrows=->,arrowscale=1.5,linewidth=0.8pt](1.2,1.5)(0.55,2.02)
\rput(0.57,1.6){$\eee_1$}
\psline[arrows=->,arrowscale=1.5,linewidth=0.8pt](2.5,1.25)(3.25,1.55)
\rput(3.35,1.3){$\eee_2$}
\rput(2.1,2.75){$M$}
\rput(3.35,0){$M_2^t$}
\rput(-0.35,0){$M_1^t$}
\psline(2,0)(2,2.5)
\psline(2.2,0)(2.2,0.2)(2,0.2)
\rput(1.8,1.2){$t$}
}}

\end{picture}
\caption{(a) $B_0$ is the line segment $MN$, and the planes $MNN_1^tM_1^t$ and $MNN_2^tM_2^t$ form the dihedral angle. (b) The triangle $M_1^tM_2^tM$ and the normal vectors to its sides.}
\label{fig4}
\end{figure} 

The vectors $\eee_1$ and $\eee_2$ are parallel to the plane of the triangle $M_1^t M_2^t M$. The sides of the triangle $M_1^t M_2^t$, $MM_1^t$, $MM_2^t$ are perpendicular to $\eee$, $\eee_1$, $\eee_2$, and their lengths are equal, respectively, to $|MM_1^t| = \lam_1 c_0 t$, $|MM_2^t| = \lam_2 c_0 t$, $|M_1^t M_2^t| = c_0 t$, where $c_0$ is a positive constants depending only on $\eee_1$, $\eee_2$, $\eee$
(namely, $c_0 = (\eee_1, \eee)/\sqrt{1 - (\eee_1, \eee)^2} + (\eee_2, \eee)/\sqrt{1 - (\eee_2, \eee)^2}$). See Fig.~\ref{fig4}\,(b).

The set $B_t$ lies in the plane $\Pi^t$ between the lines $l_1^t$ and $l_2^t$ (see Fig.~\ref{fig5}). We are going to prove that the length of the orthogonal projection of $B_t$ on $l_1^t$ is not greater than $l + o(1),$ $t \to 0$, and therefore,
\beq\label{Btl}
|B_t| \le l c_0 t(1 + o(1)), \ t \to 0.
\eeq

       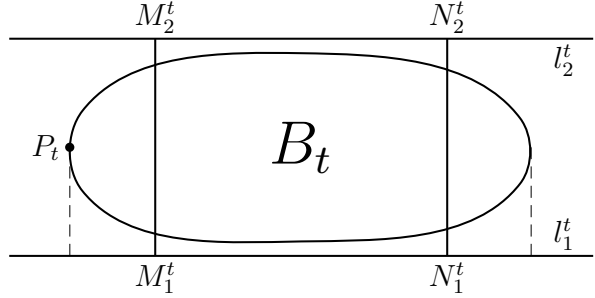
\begin{figure}[h]
\begin{picture}(0,95)
\scalebox{0.96}{
\rput(4,0.25){

\psline(0,0)(8,0)
\psline(0,3)(8,3)
\psecurve(1,2)(1,0.9)(4,0.2)(7,1)(7,1.9)(4,2.8)(1,2)(1,0.9)(4,0.2)
\psline(2,0)(2,3)
\psline(6,0)(6,3)
\rput(2,-0.3){$M_1^t$}
\rput(6,-0.3){$N_1^t$}
     \rput(7.6,0.3){$l_1^t$}
\rput(2,3.3){$M_2^t$}
\rput(6,3.3){$N_2^t$}
     \rput(7.6,2.7){$l_2^t$}
\psline[linestyle=dashed,linewidth=0.3pt](0.83,0)(0.83,1.5)
\psline[linestyle=dashed,linewidth=0.3pt](7.15,0)(7.15,1.5)
\psdots(0.83,1.5)
\rput(0.5,1.5){$P_t$}
 \rput(4,1.5){\scalebox{2}{$B_t$}}
}}

\end{picture}
\caption{The set $B_t$, the lines $l_1^t$ and $l_2^t$, and the point $P_t$.}
\label{fig5}
\end{figure} 

Assume the contrary; then there exists a sequence $(t_i)_{i \in \mathbb{N}}$ converging to 0 and a sequence of points $P_{t_i} \in B_{t_i}$ such that the distance between $P_{t_i}$ and the rectangle $M_1^{t_i} N_1^{t_i} N_2^{t_i} M_2^{t_i}$ is greater than a positive constant. The sequence $(P_{t_i})_{i \in \mathbb{N}} \subset C$ is bounded, and therefore has a limiting point $P$. This point belongs to $C$ and lies in the plane $\Pi^0 : \,(\rrr - \rrr_0,\, e) = 0$, hence $P \in B_0$. We come to the contradiction with the fact that the distance between $P$ and $B_0 = MN$ is greater than a positive constant.

Take a positive $\ve < l/2$ and denote by $M_\ve$ and $N_\ve$ the points of $B_0$ at the distance $\ve$ from $M$ and $N$, respectively. Let $\Pi^\perp_{M\ve}$ and $\Pi^\perp_{N\ve}$ be the planes through $M_\ve$ and $N_\ve$, respectively, orthogonal to $MN$. Denote by $M_{\ve 1} = M_{\ve 1}^t$, $N_{\ve 1} = N_{\ve 1}^t$, $M_{\ve 2} = M_{\ve 2}^t$, $N_{\ve 2} = N_{\ve 2}^t$ the points of intersection of these planes with the lines $l_1^t$ and $l_2^t$. Clearly, the area of the rectangle $M_{\ve 1} N_{\ve 1} N_{\ve 2} M_{\ve 2}$ is $| \square M_{\ve 1} N_{\ve 1} N_{\ve 2} M_{\ve 2}| = c_0 (l-\ve) t$.

Since both tangent cones at the points $M_\ve$ and $N_\ve$ coincide with the dihedral angle $(\rrr - \rrr_0,\, \eee_1) \le 0,\ (\rrr - \rrr_0,\, \eee_2) \le 0$ formed by the planes $\Pi_1$ and $\Pi_2$, we conclude that the tangent cone to the 2-dimensional convex body $C \cap \Pi^\perp_{M\ve}$ in the plane $\Pi^\perp_{M\ve}$ is the angle $M_{\ve 1} M_\ve M_{\ve 2}$, and the tangent cone to the body $C \cap \Pi^\perp_{N\ve}$ in the plane $\Pi^\perp_{N\ve}$ is the angle $N_{\ve 1} N_\ve N_{\ve 2}$ (see Fig.~\ref{fig6}).
           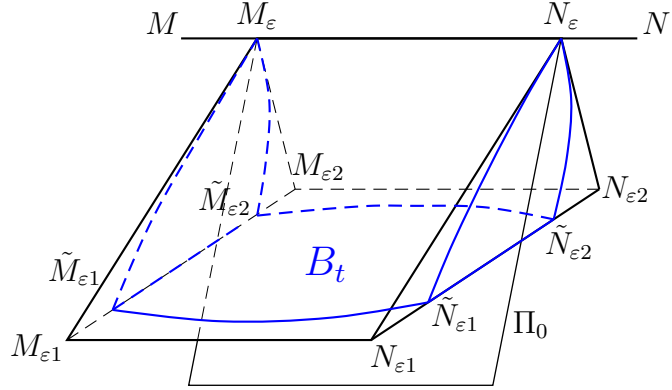
\begin{figure}[h]
\begin{picture}(0,155)
\scalebox{1}{
\rput(5,1.75){
\pspolygon(-1,-1)(3,-1)(5.5,3)(1.5,3)
     \psline[linewidth=0.5pt](0.717,-1)(0.6,-1.6)(4.6,-1.6)(5.5,3)
     \psline[linestyle=dashed,linewidth=0.25pt](0.717,-1)(1.5,3)
     \rput(5.1,-0.8){$\Pi_0$}
\psline(0.5,3)(6.5,3)
\psline(6,1)(5.5,3)
\psline(6,1)(3,-1)
\psline[linestyle=dashed,linewidth=0.4pt](1.5,3)(2,1)(6,1)
\psline[linestyle=dashed,linewidth=0.4pt](-1,-1)(2,1)
\pscurve[linecolor=blue,linestyle=dashed](1.5,0.65)(3,0.78)(3.5,0.78)(4.5,0.76)(5.4,0.6)
\pscurve[linecolor=blue](-0.4,-0.6)(1,-0.75)(2.5,-0.7)(3.75,-0.5)
\psline[linecolor=blue,linestyle=dashed](1.25,0.5)(-0.4,-0.6)
\psline[linecolor=blue](5.4,0.6)(3.75,-0.5)
   \rput(-0.9,-0.05){$\tilde M_{\ve 1}$}
   \rput(3.3,-1.25){$N_{\ve 1}$}
   \rput(2.3,1.3){$M_{\ve 2}$}
   \rput(6.34,1){$N_{\ve 2}$}
      \rput(0.25,3.2){$M$}
      \rput(6.75,3.2){$N$}
      \rput(2.4,0){\scalebox{1.3}{$\blue B_t$}}
           \pscurve[linecolor=blue,linestyle=dashed](-0.4,-0.6)(0.12,0.6)(0.8,1.8)(1.5,3)
           \pscurve[linecolor=blue,linestyle=dashed,](1.5,0.65)(1.61,1.33)(1.66,2.17)(1.5,3)
           \pscurve[linecolor=blue](3.75,-0.5)(4.08,0.3)(5.03,2.2)(5.5,3)
              \pscurve[linecolor=blue](5.4,0.6)(5.57,1.4)(5.63,2.2)(5.5,3)
\rput(1.5,3.3){$M_\ve$}
\rput(5.5,3.3){$N_\ve$}
\rput(-1.4,-1.1){$M_{\ve 1}$}
\rput(4.1,-0.65){$\tilde N_{\ve 1}$}
\rput(1.07,0.88){$\tilde M_{\ve 2}$}
\rput(5.6,0.3){$\tilde N_{\ve 2}$}
}}

\end{picture}
\caption{The part of $B_t$ contained in the rectangle $M_{\ve 1} N_{\ve 1} N_{\ve 2} M_{\ve 2}$ is bounded by the curves $\tilde M_{\ve 1} \tilde N_{\ve 1}$ and $\tilde M_{\ve 2} \tilde N_{\ve 2}$.}
\label{fig6}
\end{figure} 
 It follows that (for $t$ sufficiently small) the intersection of $C$ with the line $M_{\ve 1} M_{\ve 2}$ is a segment $\tilde M_{\ve 1} \tilde M_{\ve 2}$ contained in the segment $M_{\ve 1} M_{\ve 2}$ such that the distances $|M_{\ve 1} \tilde M_{\ve 1}|$ and $|M_{\ve 2} \tilde M_{\ve 2}|$ are $o(t)$. Similarly, for $t$ sufficiently small the intersection of $C$ with the line $N_{\ve 1} N_{\ve 2}$ is a segment $\tilde N_{\ve 1} \tilde N_{\ve 2}$ contained in the segment $N_{\ve 1} N_{\ve 2}$ and the distances $|N_{\ve 1} \tilde N_{\ve 1}|$ and $|N_{\ve 2} \tilde N_{\ve 2}|$ are $o(t)$.

The quadrangle $\tilde M_{\ve 1} \tilde N_{\ve 1} \tilde N_{\ve 2} \tilde M_{\ve 2}$ is contained in $B_t$, and its area is $|\square \tilde M_{\ve 1} \tilde N_{\ve 1} \tilde N_{\ve 2} \tilde M_{\ve 2}| = c_0 (l-\ve) t (1 + o(1))$, $t \to 0$. It follows that
\beq\label{Btg}
|B_t| \ge c_0 (l-\ve) t (1 + o(1)), \quad t \to 0.
\eeq
Since $\ve > 0$ is arbitrary, taking into account \eqref{Btl} and \eqref{Btg}, one obtains
\beq\label{Bt}
|B_t| = c_0 l t (1 + o(1)), \quad t \to 0.
\eeq

Now let us consider the surface $S_t$. The plane $\Pi_0$ of the equation $(\rrr- \rrr_0,\, \eee_1 - \eee_2) = 0$ (see Fig.~\ref{fig6}) is the bisector of the dihedral angle; it contains the segment $B_0$ and divides $S_t$ into two parts, $S_t = S_t^1 \cup S_t^2$, where $S_t^1$ lies between $\Pi_1$ and $\Pi_0$ and $S_t^2$ lies between $\Pi_2$ and $\Pi_0$.

We use the following property: if a convex body $C_1$ is contained in another convex body $C_2$ then the surface area of $C_1$ is smaller than or equal to the surface area of $C_2$; that is, $C_1 \subset C_2 \Rightarrow |\pl C_1| \le |\pl C_2|$.

Take again a positive $\ve < l/2$ and for any positive $t$ consider several prisms. Each of these prisms is bounded by the planes $\Pi_1$, $\Pi_2$, $\Pi^t$ and by two planes orthogonal to $B_0$. Let us call the {\it big prism} the smaller prism of this kind containing $C_t$. The big prism is divided into the central prism and two lateral prisms. The {\it central} prism is bounded by the orthogonal planes $\Pi^\perp_{M\ve}$ and $\Pi^\perp_{N\ve}$; of course it is contained in the big prism for $t$ small enough. The {\it lateral prisms} are obtained by taking off the central prism from the big prism. Each of them is bounded on one side by $\Pi^\perp_{M\ve}$ or $\Pi^\perp_{N\ve}$. The surface area of the central prism is asymptotically proportional to $t$; namely, it is equal to $c_0 (1 + \lam_1 + \lam_2) (l-\ve) t + O(t^2)$, where the term $O(t^2)$ is related to its lateral surface. The surface area of each of the lateral prisms is $c_0 (1 + \lam_1 + \lam_2) (\ve/2) t + o(t),\, t \to 0$.

The central prism is divided by the plane $\Pi_0$ into two parts, Prism$_{\ve,1}^t =$ Prism$_{\ve,1}$ and Prism$_{\ve,2}^t =$ Prism$_{\ve,2}$; that is, Prism$_{\ve,i}$, $i = 1,\,2$ is bounded by the planes $\Pi_{i},\, \Pi_{0},\, \Pi^t,\, \Pi^\perp_{M\ve},\, \Pi^\perp_{N\ve}$. Correspondingly, each of the surfaces $S_t^1$ and $S_t^2$ is divided into two parts, $S_t^i = S_t^{i,\ve} \cup S_t^{i,\text{lat}}$, where $S_t^{i,\ve}$ is contained in Prism$_{\ve,i}$ and $S_t^{i,\text{lat}}$ is contained in the union of lateral prisms. Note that since the surface $S_t^{i,\text{lat}}$ is contained in the union of the lateral prisms, its area does not exceed the sum of their areas $c_0 (1 + \lam_1 + \lam_2) \ve t + o(t),\, t \to 0$.

Correspondingly, we have the representation $\nu_t = \nu_{t,\ve}^1 + \nu_{t,\ve}^2 + \nu_{t,\ve}^{\text{lat}},$ and $\nu_{t,\ve}^{\text{lat}}(S^2) \le (1 + \lam_1 + \lam_2) \frac{\ve}{l} + o(1),\, t \to 0.$

Now estimate the area of $S_t^{1,\ve}$. To that end, compare the surfaces of the convex bodies Prism$_{\ve,1}$ and $C_t \cap$Prism$_{\ve,1}$. The surface measure of each of the bodies has the total integral equal to zero. The surface measure of Prism$_{\ve,1}$ is easy to calculate: it equals $c_0 t (l-\ve)(\lam_1 \del_{\eee_1} + \frac{\lam_1}{\lam_1 + \lam_2}\del_{-\eee} + \sqrt{\lam_1\lam_2} \sqrt{1 - (\lam_1 + \lam_2)^{-2}}\, \del_{\frac{\eee_2-\eee_1}{|\eee_2-\eee_1|}}) + O(t^2),\, t \to 0$. The surface measure of $C_t \cap$Prism$_{\ve,1}$ is the sum of several terms: the part of $\pl (C_t \cap$Prism$_{\ve,1})$ contained in $\Pi^t$ induces the measure $c_0 t (l-\ve) \frac{\lam_1}{\lam_1 + \lam_2}\del_{-\eee} + o(t)$; for $t$ sufficiently small the part of $\pl (C_t \cap$Prism$_{\ve,1})$ contained in $\Pi_0$ coincides with the part of $\pl $(Prism$_{\ve,1}$) contained in $\Pi_0$, and therefore induces the measure $c_0 t (l-\ve)\sqrt{\lam_1\lam_2} \sqrt{1 - (\lam_1 + \lam_2)^{-2}}\, \del_{\frac{\eee_2-\eee_1}{|\eee_2-\eee_1|}}$; the part contained in the planes $\Pi^\perp_{M\ve}$ and $\Pi^\perp_{N\ve}$ has the area $O(t^2)$; the resting part of $\pl (C_t \cap$Prism$_{\ve,1})$ is $S_t^{1,\ve}$ and is of our interest. Utilizing the integral equalities for the surface measures $\nu_{\text{Prism}_{\ve,1}}$ and $\nu_{C_t\cap\text{Prism}_{\ve,1}}$ of both convex bodies, one obtains
$$
\int_{S^2} n\, d\nu_{\text{Prism}_{\ve,1}}(n) = c_0 t (l-\ve) \Big( \lam_1 \eee_1 - \frac{\lam_1}{\lam_1 + \lam_2} \eee + \sqrt{\lam_1\lam_2} \sqrt{1 - \frac{1}{(\lam_1 + \lam_2)^{2}}}\ \frac{\eee_2-\eee_1}{|\eee_2-\eee_1|} \Big) + O(t^2) = \vec 0,
$$
$$
\int_{S^2} n\, d\nu_{C_t\cap\text{Prism}_{\ve,1}}(n) = -c_0 t (l-\ve) \frac{\lam_1}{\lam_1 + \lam_2} \eee + o(t) + c_0 t (l-\ve) \sqrt{\lam_1\lam_2} \sqrt{1 - \frac{1}{(\lam_1 + \lam_2)^{2}}}\ \frac{\eee_2-\eee_1}{|\eee_2-\eee_1|}
$$
$$
+ O(t^2) + \int_{S^2} n\, d\nu_{S_t^{1,\ve}}(n) = \vec 0.
$$
Equating these two expressions, one obtains
$$
\int_{S^2} n\, d\nu_{S_t^{1,\ve}}(n) = c_0 t (l-\ve)\lam_1\eee_1+ o(t) .
$$
Denote $c = (1 - \ve/l)\lam_1.$ After dividing both parts of this equation by $|B_t|$ one has
$$
\int_{S^2} \nnn\, d\nu_{t,\ve}^1(\nnn) = c \eee_1 + o(1), \quad t \to 0.
$$
Besides, comparing the surface areas of Prism$_{\ve,1}$ and $C_t \cap$Prism$_{\ve,1}$, one obtains $|S_{t,\ve}^1| = (l-\ve) c_0 \lam_1 t + o(t)$, therefore $\nu_{t,\ve}^1(S^2) = c + o(1),\, t \to 0.$

Let us show that for all $\del > 0$, $\nu_{t,\ve}^1(\{ \nnn \in S^2: (\nnn, \eee_1) > 1-\del \}) \to c$ and $\nu_{t,\ve}^1(\{ \nnn \in S^2: (\nnn, \eee_1) \le 1-\del \}) \to 0$ as $t \to 0$; it will follow that $\nu_{t,\ve}^1$ weakly converges to $c\del_{\eee_1}$.

We have
$$
\int_{S^2} (1 - (\nnn, \eee_1))\, d\nu_{t,\ve}^1(\nnn) = \nu_{t,\ve}^1(S^2) - \Big( \int_{S^2} \nnn\, d\nu_{t,\ve}^1(\nnn),\,  \eee_1 \Big)= o(1), \quad t \to 0.
$$
By Chebyshev's inequality,
$$
\nu_{t,\ve}^1(\{ \nnn \in S^2: (\nnn, \eee_1) \le 1-\del \}) \le \frac{1}{\del} \int_{S^2} (1 - (\nnn, \eee_1))\, d\nu_{t,\ve}^1(\nnn) \to 0, \ \text{as} \ t \to 0.
$$
Hence $\nu_{t,\ve}^1(\{ \nnn \in S^2: (\nnn, \eee_1) > 1-\del \}) = \nu_{t,\ve}^1(S^2) - \nu_{t,\ve}^1(\{ \nnn \in S^2: (\nnn, \eee_1) \le 1-\del \}) = c + o(1),\, t \to 0.$

Thus, it is proved that $\nu_{t,\ve}^1 \xrightarrow[t\to 0]{}(1 - \ve/l)\lam_1\del_{\eee_1}$. In a similar way one obtains that $\nu_{t,\ve}^2 \xrightarrow[t\to 0]{}(1 - \ve/l)\lam_2\del_{\eee_2}$.

Take a continuous function $f$ on $S^2$. We have
$$
\Big| \int_{S^2} f(n)\, d\nu_t(n) - (\lam_1 f(\eee_1) + \lam_2 f(\eee_2)) \Big| \le
$$
$$
\le \Big| \int_{S^2} f(n)\, d\nu_{t,\ve}^1(n) - \lam_1 f(\eee_1) \Big|
+ \Big| \int_{S^2} f(n)\, d\nu_{t,\ve}^2(n) - \lam_2 f(\eee_2) \Big| + \Big| \int_{S^2} f(n)\, d\nu_{t,\ve}^{\text{lat}}(n) \Big|.
$$
The first term in the right hand side of this inequality converges to $\frac{\ve}{l}\, \lam_1 |f(\eee_1)|$, the second term converges to $\frac{\ve}{l}\, \lam_2 |f(\eee_2)|$, and the upper limit of the third term is not greater than $\frac{\ve}{l}\, (1 + \lam_1 + \lam_2) \cdot \max f$. Since $\ve > 0$ is arbitrary, one concludes that the integral in the left hand side of the inequality converges to zero. Thus, $\nu_t$ weakly converges to $\lam_1 \del_{\eee_1} + \lam_2 \del_{\eee_2}$ as $t \to 0$. Theorem \ref{t2 segment} is proved.


\section{Proof of Theorem \ref{t1}}\label{sect t1}

If $B_0$ is a non-degenerate line segment then, by Theorem \ref{t2 segment}, the unique partial limit of $\nu_t$ is the measure $\lim_{t\to0} \nu_t = \nu_* = \lam_1 \del_{\eee_1} + \lam_2 \del_{\eee_2}$ supported on the two-point set $\{ \eee_1,\, \eee_2 \}$, and the theorem is proved. It remains to consider the case when $B_0$ is the singleton, $B_0 = \{ \rrr_0 \}$.

\begin{lemma}\label{l1}
$\lim\sup_{t\to0} \frac{|S_t|}{|B_t|} \le$ \rm const.
\end{lemma}

\begin{proof}
Consider the projection in the direction $\eee_1 + \eee_2$ on the plane $\Pi^t$. Let us show that for $t$ sufficiently small, the restriction of the projection on $S_t$ is injective and the image of $S_t$ is $B_t$.

Indeed, all elements of the arc $\Gam$ have the form $n = \mu_1 \eee_1 + \mu_2 \eee_2$, $\mu_1 \ge 0$, $\mu_2 \ge 0$, and therefore satisfy the inequality $(n,\, \eee_1 + \eee_2) > 0$. Hence there exists a neighborhood $\NNN$ of this arc such that all vectors $n \in \NNN$ satisfy the same inequality. It follows that for $t$ sufficiently small and for all regular vectors $\xi \in S_t$ it holds $(n_\xi,\, \eee_1 + \eee_2) > 0$, and therefore, the restriction of the projection on $S_t$ is injective.

Let us show that for $t$ sufficiently small, the image of $S_t$ belongs to $B_t$. Take the image $\rrr + \bar s (\eee_1 + \eee_2)$ of a point $\rrr \in S_t$. First note that $\bar s \le 0$. Indeed, one has $(\rrr + \bar s (\eee_1 + \eee_2) - \rrr_0,\, \eee) = -t$ and $(\rrr - \rrr_0,\, \eee) \ge -t$, hence $\bar s (\eee_1 + \eee_2,\, \eee) = -t - (\rrr - \rrr_0,\, \eee) \le 0,$ and therefore, $\bar s \le 0$. Assume that the image $\rrr + \bar s (\eee_1 + \eee_2)$ lies in $\Pi_t \setminus B_t$, and thereby, does not belong to $C$. It follows that the segment $\rrr + s(\eee_1 + \eee_2)$, $\bar s \le s \le 0$ intersects $\pl C$ at some interior point $\xi = \rrr + s(\eee_1 + \eee_2)$, and at this point $(n_\xi,\, \eee_1 + \eee_2) \le 0$. For $t$ sufficiently small this is impossible, and therefore, the image of $\rrr$ belongs to $B_t$.

Let us now show that $B_t$ is contained in the image of $S_t$.
Take a point $\rrr \in B_t$ and consider $\rrr_{(s)} = \rrr + s(\eee_1 + \eee_2)$ for $s \ge 0$. From the inequalities $(\rrr_{(s)} - \rrr_0,\, \eee_1) = (\rrr - \rrr_0,\, \eee_1) + s(1 + (\eee_1,\eee_2))$ and $(\rrr_{(s)} - \rrr_0,\, \eee_2) = (\rrr - \rrr_0,\, \eee_2) + s(1 + (\eee_1,\eee_2))$ it follows that for $s$ sufficiently large both these values are positive, and therefore, $\rrr_{(s)}$ does not belong to $C$. Hence for some $\bar s \ge 0$, $\rrr_{(\bar s)}$ lies on $\pl C$. We have $(\rrr_{(\bar s)} - \rrr_0,\, \eee) = -t + \bar s(\eee_1+\eee_2, \eee) \ge -t$, that is, $\rrr_{(\bar s)}$ lies on $S_t$. It follows that the image of $S_t$ is $B_t$.

The area of $B_t$ is greater or equal than the area of the (orthogonal) projection $\text{proj}_{\eee_1+\eee_1} (S_t)$ of $S_t$ on the plane $(\rrr,\, \eee_1 + \eee_2) = 0$, $|B_t| \ge |\text{proj}_{\eee_1+\eee_1} (S_t)|$. On the other hand,
$$
\frac{|S_t|}{|\text{proj}_{\eee_1+\eee_1} (S_t)|} \le \frac{1}{\inf_{\xi\in S_t} \big( \nnn_\xi,\ \frac{\eee_1+\eee_2}{|\eee_1+\eee_2|} \big)}.
$$
Since $S_t$ is compact, the value in the denominator in the right hand side of this inequality is positive.
Further, we have
$$
\lim_{t\to0} \inf_{\xi\in S_t} \Big( \nnn_\xi,\ \frac{\eee_1+\eee_2}{|\eee_1+\eee_2|} \Big)
= \inf_{\nnn \in \Gam_{\eee_1,\eee_2}} \Big( \nnn,\ \frac{\eee_1+\eee_2}{|\eee_1+\eee_2|} \Big)
= \Big( \eee_1,\ \frac{\eee_1+\eee_2}{|\eee_1+\eee_2|} \Big) = \frac{\sqrt{1 + (\eee_1,\eee_2)}}{\sqrt{2}} > 0.
$$
It follows that
$$
\lim\sup_{t\to0} \frac{|S_t|}{|B_t|} \le \lim\sup_{t\to0} \frac{|S_t|}{|\text{proj}_{\eee_1+\eee_1} (S_t)|} \le \frac{\sqrt{2}}{\sqrt{1 + (\eee_1,\eee_2)}}.
$$
\end{proof}

\begin{corollary}
It follows from Lemma \ref{l1} that $\nu_t(S^2) \le$ const for all $t$, and therefore, there exists at least one partial limit of $\nu_t$. If the partial limit $\nu_*$ is unique then $\nu_t$ converges to $\nu_*$.
\end{corollary}

\begin{lemma}\label{l sup}
Each partial limit of $\nu_t$ as $t \to 0$ is supported in $\Gam.$
\end{lemma}

\begin{proof}
Notice that the intersection of the nested family of closed sets $S_t$ is the point $\rrr_0$, $\cap_{t>0} S_t = \{ \rrr_0 \}$, therefore for all $\ve > 0$ there exists $t_0$ such that for any $t < t_0$, $S_t$ is contained in the $\ve$-neighborhood of $\rrr_0$.

For a set of unit vectors $K \subset S^2$, denote by $\OOO_K$ the set of points $\rrr \in \pl C$ such that there exists a plane of support to $C$ at $\rrr$ with the outward normal contained in $K$. If $K$ is compact then $\OOO_K$ is also compact. Indeed, let $\rrr_i,\, i = 1,\, 2,\ldots$ be a sequence of points from $\OOO_K$ converging to a point $\rrr$. For all $i$ there exists a plane of support at $\rrr_i$ with the outward normal $v_i \in K$. Without loss of generality assume that $v_i$ converges to a vector $v$; otherwise just take a converging subsequence of vectors. Then we have $v \in K$, and the plane through $\rrr$ with the outward normal $v$ is a plane of support. We conclude that $\rrr$ is contained in $\OOO_K$, and therefore, $\OOO_K$ is closed, and thereby, is compact.

Take an open set $\NNN \subset S^2$ containing $\Gam$. The set $\OOO_{S^2 \setminus \NNN} \subset \pl C$ is closed and does not contain $\rrr_0$. Then there exists $t_0 > 0$ such that for all $t < t_0$,\, $S_t$ is contained in $\pl C \setminus \OOO_{S^2 \setminus \NNN}$. It follows that the outward normals to all planes of support at points of $S_t$ are contained in $\NNN$; therefore $\nu_t$ is supported in the closure of $\NNN$. Since the set $\NNN$ containing $\Gam$ is arbitrary, we conclude that all partial limits of $\nu_t$ as $t \to 0$ are supported in $\Gam.$
\end{proof}

It remains to prove the following statement.

\begin{lemma}\label{l3}
The support of any partial limit of $\nu_t$ contains the points $\eee_1$ and $\eee_2$.
\end{lemma}

\begin{proof}
It suffices to prove the lemma for the point $\eee_1$.

Recall that $\Pi_1$ and $\Pi_2$ are the planes of equations $(\rrr - \rrr_0,\, \eee_1) = 0$ and $(\rrr - \rrr_0,\, \eee_2) = 0$ forming the dihedral angle, $\Pi^{t}$ is the plane of equation $(\rrr - \rrr_0,\, \eee) = -t$, and $l_1 = l_1^t$ and $l_2 = l_2^t$ are the parallel lines resulting from intersection of $\Pi_1$ and $\Pi_2$, correspondingly, with $\Pi^t$.

Let $\Pi^\perp$ be the plane through $\rrr_0$ perpendicular to the edge of the dihedral angle, and let this plane intersects the lines $l_1$ and $l_2$ at the points $\hat N_t$ and $\hat M_t$, respectively. We have $|\hat M_t \hat N_t| = c_0 t$, where $c_0$ is the positive constant defined above.

For $t$ sufficiently small, the intersection of $\Pi^\perp$ with $B_t$ is a line segment, let its endpoints be $M = M_t$ and $N = N_t$. This segment is contained in $\hat M_t \hat N_t$, and $|M_t N_t|/|\hat M_t \hat N_t| \to 1$ as $t \to 0.$

The domain $B_t$ is contained between the lines $l_1$ and $l_2$. Let the orthogonal projection of $B_t$ on the line $l_1$ be the segment $\hat C_t \hat D_t.$ For $t$ sufficiently small this segment contains $\hat N_t$, and $|\hat C_t \hat N_t|/t \to \infty$, $|\hat D_t \hat N_t|/t \to \infty$ as $t \to 0.$ Without loss of generality assume that $|\hat D_t \hat N_t| \ge |\hat C_t \hat N_t|$. Denote $|\hat D_t \hat N_t| = d(t)$; thus, $\lim_{t\to0} d(t)/t = \infty.$

One easily sees that
\beq\label{estimateB}
|B_t| \le 2c_0 t d(t).
\eeq

Fix a value $0 < \theta < 1$ and take the point $\hat F_t$ on $l_1$ between the points $\hat N_t$ and $\hat D_t$ so as $|\hat N_t \hat F_t|/|\hat N_t \hat D_t| = \theta.$ Thus, $|\hat N_t \hat F_t| = \theta d(t)$ and $|\hat F_t \hat D_t| = (1-\theta) d(t)$.

Let $D_t$ be the point on $\pl B_t$ that projects to $\hat D_t$; if such point is not unique (these points form a line segment) then choose any of such points, for instance the one closest to $l_1$. Denote by $\Pi_{[t]}$ the plane through $D_t$ and the edge of the dihedral angle, and denote by $H_t$ the point of intersection of this plane with $M_t N_t$.   

Denote $|H_t \hat N_t| := h(t) \le c_0 t$. It may happen that the segment $H_t \hat N_t$ degenerates to a point; in this case $h(t) = 0$. One has $|H_t N_t| = h(t) + o(t)$ as $t \to 0$.

The area of the rectangle $H_t \hat N_t \hat F_t E_t$ equals
$$
|H_t \hat N_t \hat F_t E_t| = \theta d(t) h(t).
$$

Denote by $B_t^{(1)}(\theta)$ the part of $B_t$ bounded by the broken line $N_t H_t E_t F_t$ and by the arc $N_t F_t$ of the boundary $\pl B_t.$ Due to convexity of the curve $N_t F_t$, one easily concludes that $|E_t F_t| \ge (1-\theta) |H_t N_t|$, and therefore
$$
 \theta d(t) \big( (1 - \theta/2)h(t) + o(t) \big) \le |B_t^{(1)}(\theta)| \le  \theta d(t) h(t).
$$
It follows that for some $\theta_1 = \theta_1(t) \in [0,\, \theta]$,
\beq\label{estimate1}
|B_t^{(1)}(\theta)| =  \theta d(t) \big( (1 - \theta_1/2)h(t) + o(t) \big).
\eeq
See Figures~\ref{figs1} and \ref{figs2}.  

       \begin{figure}[h]
\begin{picture}(0,150)
\scalebox{1}{
\rput(0.8,1){
\psline(-0.5,0)(14.5,0)
\psline(-0.5,4)(14.5,4)
\pspolygon[linewidth=0pt,linecolor=white,fillstyle=solid,fillcolor=lightgray](6,0.12)(9.5,0.265)(9.5,1.75)(6,1.75)
\psecurve(-0.05,1)(0,2.35)(0.5,3)(6,3.9)(12,3.25)(13.8,1.75)(14.3,0.1)
\psecurve(14.6,4)(13.8,1.75)(13.5,1.29)(8,0.2)(0.8,1.1)(0,2.35)(0.3,3.9)
\psline(6,0)(6,4)
\rput(5.9,-0.3){$\hat N_t$}
     \rput(5.73,0.38){$N_t$}
\rput(6,4.3){$\hat M_t$}
     \rput(5.7,3.6){$M_t$}
\rput(14.1,1.75){$D_t$}
\rput(13.8,-0.3){$\hat D_t$}
\psline[linestyle=dashed,linewidth=0.3pt](13.8,1.75)(13.8,0)
\rput(0,-0.3){$\hat C_t$}
\psline[linestyle=dashed,linewidth=0.3pt](0,2.35)(0,0)
\rput(-0.25,2.5){$C_t$}
\psline(6,1.75)(13.8,1.75)
\rput(5.7,1.8){$H_t$}
\rput(2,-0.3){$l_1^t$}
\rput(2,4.3){$l_2^t$}
\psline(9.5,0)(9.5,1.75)
\rput(9.5,-0.3){$\hat F_t$}
\rput(9.6,2){$E_t$}
\rput(9.75,0.56){$F_t$}
\rput(7.85,1){\scalebox{1.5}{$B_t^{(1)}(\theta)$}}
\rput(7.85,-0.3){$\theta d(t)$}
\rput(11.85,-0.3){$(1-\theta) d(t)$}
\rput(5.55,1){$h(t)$}
}}
\end{picture}
\caption{The domain $B_t$ between the lines $l_1^t$ and $l_2^t$ and the corresponding notation.}
\label{figs1}
\end{figure}
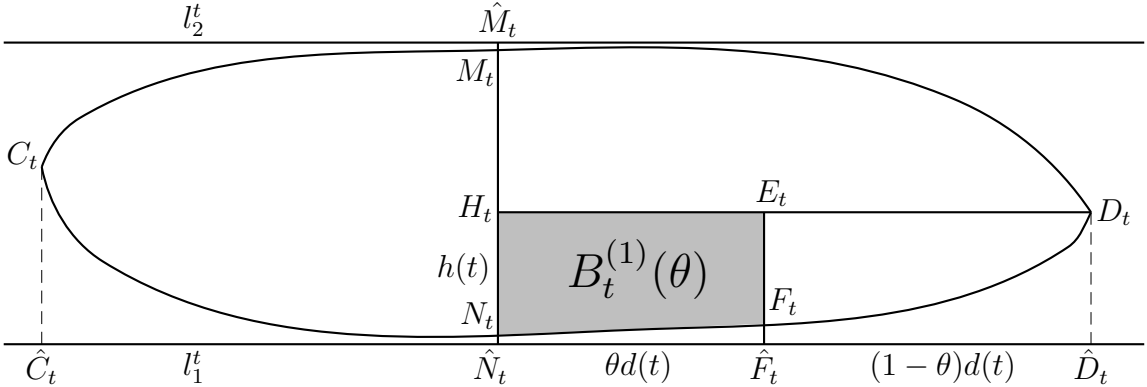

\begin{figure}[h]
\begin{picture}(0,180)
\scalebox{1}{
\rput(0.8,0.5){
\psline(2,0)(1.5,5)
\psline[linestyle=dashed,linewidth=0.3pt](0,2)(12,2)(14,0)
\psline[linestyle=dashed,linewidth=0.3pt](12,2)(13.5,5)
\pspolygon(2,0)(0,2)(1.5,5)(13.5,5)(14,0)
\psline[linestyle=dashed,linewidth=0.3pt](7,0)(5,2)(6.5,5)
\psline[linewidth=0.8pt](7,0)(6.5,5)
\psline[linestyle=dashed,linewidth=0.3pt](6,1)(13,1)
\pscurve[linecolor=brown](1,1)(2.5,3.25)(4.5,4.7)(6.5,5)(8.5,4.83)(10.4,4.3)(11,3.92)(12,3)(13,1)
\psecurve[linestyle=dashed,linecolor=brown,linewidth=0.8pt]
(5,0.3)(7,0.23)(9,0.3)(11,0.35)(12,0.44)(13,0.95)(12.95,1)(12.45,1.35)(11,1.72)(9,1.9)
(5,1.9)(3,1.8)(1.54,1.5)(1.06,0.95)(1.54,0.6)(3,0.4)(5,0.3)(7,0.23)(9,0.3)
\psline[linestyle=dashed,linewidth=0.3pt](11,0)(9,2)(10.5,5)
          \psline[arrows=->,arrowscale=2,linewidth=0.8pt](10.17,2.3)(11.07,1.4) 
          \rput(11.3,1.4){$\vvv$}
\psline[linewidth=0.8pt](11,0)(10.5,5)
\psline[linestyle=dashed,linewidth=0.3pt](6,1)(6.5,5)
\psline[linestyle=dashed,linewidth=0.3pt](10,1)(10.5,5)
\psline[linestyle=dashed,linewidth=0.3pt](13,1)(13.5,5)
\psdots[dotsize=2pt](6.77,0.23)(10.4,4.3)(10.66,0.34)
\psecurve[linecolor=brown,linewidth=0.8pt](6.19,7.5)(6.5,5)(6.69,2.5)(6.77,0.23)(6.9,-2.5)
\pscurve[linecolor=brown,linewidth=0.8pt](10.4,4.3)(10.6,2.3)(10.67,1.3)(10.68,0.82)(10.66,0.34)
\pscurve[linecolor=brown,linewidth=0.8pt](6.77,0.23)(8.71,0.28)(10.66,0.34)
\rput(6.5,5.27){$\rrr_0$}
     \rput(6.24,3.4){$\Pi^\perp$}
\rput(7.1,-0.33){$\hat N_t$}
\rput(11.1,-0.33){$\hat F_t$}
\rput(9.85,0.75){$E_t$}
\rput(5.65,0.95){$H_t$}
\rput(4.7,2.3){$\hat M_t$}
      \rput(4.95,1.64){$M_t$}
\rput(10.5,5.33){$\rrr_{\theta,t}$}
\rput(10.02,4.12){\scalebox{0.9}{$q_\theta(t)$}}
    \rput(9.77,2.7){\scalebox{0.9}{$\Pi_\theta^\perp(t)$}}
\rput(6.35,0.25){$N_t$}
\rput(10.45,0.58){\scalebox{1}{$F_t$}}
}}
\end{picture}
\caption{The convex body $C_t$ in the dihedral angle and the corresponding notation.}
\label{figs2}
\end{figure}
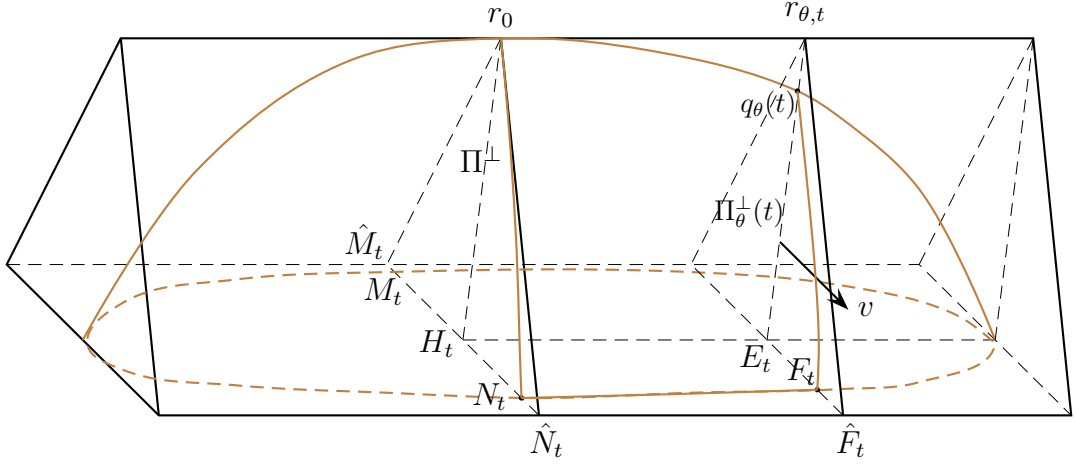 

Draw the plane $\Pi^\perp_\theta(t)$ through $E_t$ perpendicular to the edge of the dihedral angle. Let $\rrr_{\theta,t}$ be the point of intersection of the plane with the edge, and let the segment $E_t \rrr_{\theta,t}$ intersect $\pl C$ at the point $q_{\theta,t}$.

Denote by $\vvv = \vvv(t)$ the normal vector to $\Pi_{[t]}$ directed toward $l_1.$

Denote by $C_t(\theta)^{(1)}$ the part of $C$ bounded by the planes $\Pi_1,\, \Pi^t,\, \Pi_{[t]},\, \Pi^\perp,$ and $\Pi^\perp_\theta(t)$, and compare it with the prism with the bases $H_t \hat N_t \rrr_0$ and $E_t \hat F_t \rrr_{\theta,t}$. Actually, our aim is to compare the part of $\pl C$ bounded by the curve $N_t F_t q_{\theta,t} \rrr_0$ (which belongs to the boundary of the body $C_t(\theta)^{(1)}$) and the rectangle $\rrr_0 \rrr_{\theta,t} \hat F_t \hat N_t$ (which belongs to the boundary of the prism).

Let $|\rrr_0 H_t| = c_1 t$; then $| \square \rrr_0 H_t E_t \rrr_{\theta,t}| = \theta d(t) c_1 t$. The area of the curvilinear quadrangle $\rrr_0 H_t E_t q_{\theta,t}$ obeys the inequalities
$$
\theta d(t) (1 - \theta/2) c_1 t  \le |\square \rrr_0 H_t E_t q_{\theta,t}| \le |\square \rrr_0 H_t E_t \rrr_{\theta,t}| = \theta d(t) c_1 t.
$$
It follows that for some $\theta_2 = \theta_2(t) \in [0,\, \theta]$,
\beq\label{estimate2}
|\square \rrr_0 H_t E_t q_{\theta,t}| = \theta d(t) (1 - \theta_2/2) c_1 t.
\eeq

Further, one has
$$
|\square\, \rrr_0 \rrr_{\theta,t} \hat F_t \hat N_t|\, \eee_1 = \theta d(t) c_1 t\, \vvv + \theta d(t) h(t)\, \eee,
$$
hence
$$
e_1 = \frac{\vvv + \frac{h(t)}{c_1 t}\, \eee}{|\vvv + \frac{h(t)}{c_1 t}\, \eee|}.
$$

Denote by $S_{t,\theta}^{1}$ the part of the surface $\pl C$ bounded by $N_t F_t q_{\theta,t} \rrr_0$ and consider the corresponding induced measure $\nu_{S_{t,\theta}^{1}}$. Note that the outward normals to the body $C_t(\theta)^{(1)}$ at the points of its boundary contained in the planes $\Pi^\perp$ and $\Pi_\theta^\perp(t)$ are, respectively, $\frac{\eee_1 \times \eee_2}{|\eee_1 \times \eee_2|}$ and $-\frac{\eee_1 \times \eee_2}{|\eee_1 \times \eee_2|}$. Then the integral equality $\int_{S^2} n\, d\nu_{C_t(\theta)^{(1)}}(n) = \vec 0$ for the surface measure of the body $C_t(\theta)^{(1)}$ can be rewritten as
$$
\int_{S^2} n\, d\nu_{S_{t,\theta}^{1}}(n) = |B_t^{(1)}(\theta)|\, e + |\square \rrr_0 H_t E_t q_{\theta,t}|\, \vvv
+ |\rrr_0 H_t N_t|\, \frac{\eee_1 \times \eee_2}{|\eee_1 \times \eee_2|} - |q_{\theta,t} E_t F_t|\, \frac{\eee_1 \times \eee_2}{|\eee_1 \times \eee_2|}.
$$
Using \eqref{estimate1} and \eqref{estimate2} and taking into account that $|\rrr_0 H_t N_t| = O(t^2)$, $|q_{\theta,t} E_t F_t| = O(t^2)$,  and $\frac{O(t^2)}{td(t)} = o(1)$ as $t \to 0$, one gets
$$
\frac{1}{|B_t|} \int_{S^2} n\, d\nu_{S_{t,\theta}^{1}}(n) = \frac{\theta\, d(t)\, c_1 t}{|B_t|} \Big[ \big( (1 - {\theta_1}/{2})   \frac{h(t)}{c_1 t} + o(1) \big)\, \eee
+ \Big( 1 - \frac{\theta_2}{2} \Big)\, \vvv + o(1) \Big]
$$
\beq\label{eqTheta}
= \frac{\theta\, d(t)\, c_1 t}{|B_t|} \left( \Big[\vvv + \frac{h(t)}{c_1 t}\, \eee \Big] - \Big[\frac{\theta_2}{2}\, \vvv + \frac{\theta_1}{2}\, \frac{h(t)}{c_1 t}\, \eee \Big] + o(1) \right).
\eeq

The angle between $e_1$ and $\int_{S^2} n\, d\nu_t^\theta(n)$ is equal to the angle between $\vvv + \frac{h(t)}{c_1 t}\, \eee$ and $[\vvv + \frac{h(t)}{c_1 t}\, \eee] - [\frac{\theta_2}{2}\, \vvv + \frac{\theta_1}{2}\, \frac{h(t)}{c_1 t}\, \eee] + o(1)$. The angle between the unit vectors $\vvv$ and $\eee$ is always greater than the minimum of the angles $\measuredangle(\eee,-\eee_1)$ and $\measuredangle(\eee,\eee_2)$, therefore the norm of the vector $\vvv + \frac{h(t)}{c_1 t}\, \eee$ is greater than a certain positive value $C_1$ depending only on $\eee,\, \eee_1$, and $\eee_2$,
\beq\label{eq1}
\big|\vvv + \frac{h(t)}{c_1 t}\, \eee\big| > C_1.
\eeq
On the other hand, since $\frac{h(t)}{c_1 t} \le \frac{c_0}{c_1} \le \max\big\{ \frac{|\hat M_t \rrr_0|}{|\hat M_t \hat N_t|},\, \frac{|\hat N_t \rrr_0|}{|\hat M_t \hat N_t|}  \big\} = \max \{ \lam_1,\, \lam_2 \}$, one concludes that the norm of $\frac{\theta_2}{2}\, \vvv + \frac{\theta_1}{2}\, \frac{h(t)}{c_1 t}\, \eee$ is smaller than $\theta$ times a  certain positive value $C_2$ depending only on $\eee,\, \eee_1$, and $\eee_2$,
\beq\label{eq2}
\big|\frac{\theta_2}{2}\, \vvv + \frac{\theta_1}{2}\, \frac{h(t)}{c_1 t}\, \eee\big| < C_2 \theta.
\eeq

Further, using \eqref{estimateB}, one has
$$
\frac{\theta\, d(t)\, c_1 t}{|B_t|} \ge \frac{\theta\, c_1}{2\, c_0},
$$
hence the normalized total norm $\frac{1}{|B_t|}\, \nu_{S_{t,\theta}^{1}}(S^2)$ satisfies the inequality
$$
\frac{1}{|B_t|}\, \nu_{S_{t,\theta}^{1}}(S^2) \ge \frac{1}{|B_t|}\, \Big| \int_{S^2} n\, d\nu_{S_{t,\theta}^{1}}(n) \Big|
\ge \frac{\theta\, c_1}{2\, c_0} \big( C_1 - C_2 \theta + o(1) \big) \quad \text{as} \ \ t \to 0,
$$
and therefore, for $\theta < C_1/C_2$ each partial limit of $\frac{1}{|B_t|}\, \nu_{S_{t,\theta}^{1}}$ is nonzero.

Let $\nu_*$ be a partial limit of $\nu_t$, and let the sequence $t_i$ converging to zero be such that $\nu_* = \lim_{i \to \infty} \nu_{t_i}$. Equation \eqref{eqTheta} can be rewritten as
$$
\frac{1}{|B_t|} \int_{S^2} n\, d\nu_{S_{t,\theta}^{1}}(n) = \frac{\theta\, d(t)\, c_1 t}{|B_t|} \left( \big|\vvv + \frac{h(t)}{c_1 t}\, \eee \big| \eee_1 - \Big[\frac{\theta_2}{2}\, \vvv + \frac{\theta_1}{2}\, \frac{h(t)}{c_1 t}\, \eee \Big] + o(1) \right)
$$
$$
= \frac{\theta\, d(t)\, c_1 t}{|B_t|}\ \big|\vvv + \frac{h(t)}{c_1 t}\, \eee \big|
\left( \eee_1\ - \ \frac{\frac{\theta_2}{2}\, \vvv + \frac{\theta_1}{2}\, \frac{h(t)}{c_1 t}\, \eee + o(1)}{\big|\vvv + \frac{h(t)}{c_1 t}\, \eee \big| } \right), \quad t \to 0.
$$
Taking into account inequalities \eqref{eq1} and \eqref{eq2}, one concludes that the center of mass of each partial limit of $\frac{1}{|B_{t_i}|}\, \nu_{S_{{t_i},\theta}^{1}}$ is contained in the convex cone $\UUU_\theta(\eee_1) = \{ \lam(\eee_1 + \theta \frac{C_2}{C_1}\, w),\, |w| \le 1,\, \lam > 0 \}$ shrinking to the ray $\{ \lam\eee_1,\, \lam > 0 \}$ as $\theta \to 0$ (and of course the set of such partial limits is nonempty). Since $\frac{1}{|B_t|}\, \nu_{S_{t,\theta}^{1}} \le \nu_t$ and $\lim_{i \to \infty} \nu_{t_i} = \nu_*$ is supported on $\Gam$, each partial limit of $\frac{1}{|B_{t_i}|}\, \nu_{S_{{t_i},\theta}^{1}}$ is also supported on $\Gam$, and the support has nonempty intersection with $\UUU_\theta(\eee_1)$ (otherwise its center of mass does not belong to $\UUU_\theta(\eee_1)$). It follows that spt$\,\nu_*$ also has nonempty intersection with $\UUU_\theta(\eee_1)$.

Since $\theta$ can be made arbitrarily small, spt$\,\nu_*$ contains points that are arbitrarily close to $\eee_1$, and therefore, contains $\eee_1$.
\end{proof}

\section{Proof of Theorem \ref{t1b}}\label{sect t1b}

Let $\al$ be the angle between $\eee$ and $\eee_1$ and $\bt$ be the angle between $\eee$ and $\eee_2$; that is, $(\eee, \eee_1) = \cos\al$ and  $(\eee, \eee_2) = \cos\bt$, $\al > 0$, $\bt > 0$, $\al + \bt < \pi$. Take the orthogonal coordinate system with the coordinates $x,\, y,\, z$ so as the origin is at $\rrr_0$ and the vectors $\eee,\, \eee_1,\, \eee_2$ take the form $\eee = (0,0,-1)$, $\eee_1 = (-\sin\al, 0, -\cos\al)$, $\eee_2 = (\sin\bt, 0, -\cos\bt)$. One then has $\lam_1 = \sin\bt/\sin(\al+\bt)$, $\lam_2 = \sin\al/\sin(\al+\bt)$, $\rrr_0 = (0,0,0)$, and equations \eqref{dih angle} for the tangent cone are transformed into
\beq\label{dih2}
-z\cot\al \le x \le z\cot\bt.
\eeq
Note that $\cot\al + \cot\bt > 0$, and therefore, the cone is contained in the positive half-space $z \ge 0$.

The formulas for $C_t$, $B_t$, and $S_t$ take the form
$$C_t = C \cap \{ z \le t \}, \quad B_t = C \cap \{ z = t \}, \ \ \text{ and } \ \ S_t = \pl C \cap \{ z \le t \}.
$$

The image of the arc $\Gam$ under the orthogonal projection on the $xz$-plane is the arc in $S^1$ with the endpoints $(-\sin\al,-\cos\al)$ and $(\sin\bt, -\cos\bt)$. We identify the points $(\sin\vphi, -\cos\vphi)$ of the circumference $S^1$ with the values $\vphi \in [-\pi/2,\, \pi/2]$. In particular, the image of $\Gam$ is identified with the segment $[-\al,\, \bt]$. The image of the compact set $K \subset \Gam$ is identified with a subset of $[-\al,\, \bt]$, which will also be denoted by $K$. That is, with this identification we have $\{ -\al,\, \bt \} \subset K \subset [-\al,\, \bt]$.

Now choose a finite measure $\mu$ on $S^1$ supported on $K$. One can, for instance, define the generating function of this measure $F_\mu(\vphi) = \mu([-\al, \vphi])$ as follows. The open set $[-\al,\, \bt] \setminus K$ is the finite or countable union of disjoint open intervals, $[-\al,\, \bt] \setminus K = \cup_i (a_i,\, b_i)$. To each $i$ assign a value $\xi_i \in (a_i,\, b_i)$. For $\vphi \in (a_i,\, b_i)$ define $F_\mu(\vphi) = \al + \xi_i$. Then extend the definition of $F_\mu$ to the closure $\overline{\cup_i (a_i,\, b_i)}$, so as the resulting function is monotone increasing and right semi-continuous. Finally, define $F_\mu$ on the set $[-\al, \bt] \setminus \overline{\cup_i (a_i,\, b_i)}$ (which is again the disjoint union of intervals) so as it is affine, nonnegative, and strictly monotone increasing on the closure of each interval.

Denote by $V$ the angle $-z\cot\al \le x \le z\cot\bt$.  We are going to find a planar curve inducing the measure $\mu$ with the endpoints on the sides of this angle $x = -z\cot\al,\, z \ge 0$ and $x = z\cot\bt,\, z \ge 0$ such that the unbounded domain (denoted by $\BBB$) with the boundary composed of the curve and the rays is convex. See Fig.~\ref{figCurve}\,(a,b).

       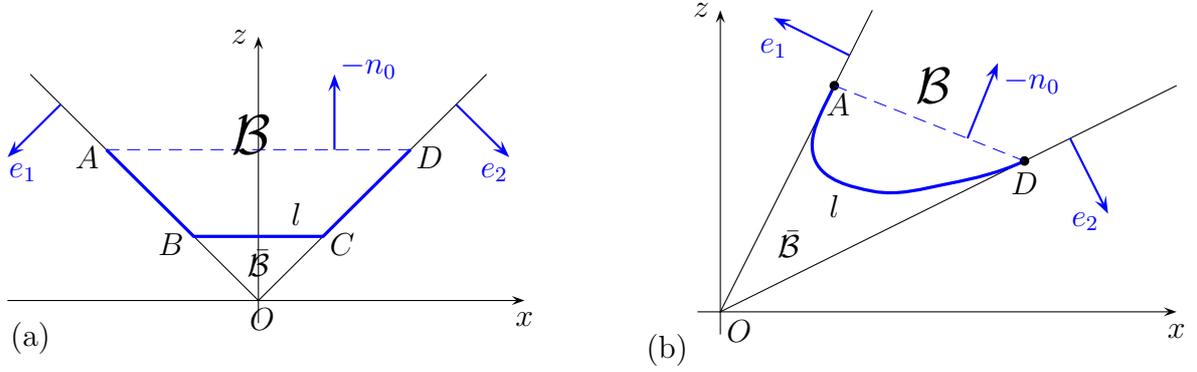
\begin{figure}[h]
\begin{picture}(0,140)
\scalebox{1}{
\rput(3.4,0.75){
\psline[arrows=->,arrowscale=1.5,linewidth=0.4pt](-3.3,0)(3.5,0)
\psline[arrows=->,arrowscale=1.5,linewidth=0.4pt](0,-0.3)(0,3.5)
\psline[linewidth=0.4pt](-3,3)(0,0)(3,3)
\psline[linewidth=1.2pt,linecolor=blue](-2,2)(-0.85,0.85)(0.85,0.85)(2,2)
          \psline[linewidth=0.3pt,linestyle=dashed,linecolor=blue](-2,2)(2,2)
     \psline[linecolor=blue,arrows=->,arrowscale=1.5,linewidth=0.8pt](-2.6,2.6)(-3.3,1.9)
     \psline[linecolor=blue,arrows=->,arrowscale=1.5,linewidth=0.8pt](2.6,2.6)(3.3,1.9)
          \psline[linecolor=blue,arrows=->,arrowscale=1.5,linewidth=0.8pt](1,2)(1,3)
          \rput(1.45,3.1){$\blue -n_0$}
     \rput(-3.1,1.7){$\blue \eee_1$}
     \rput(3.1,1.7){$\blue \eee_2$}
\rput(3.5,-0.25){$x$}
\rput(-0.25,3.5){$z$}
\rput(-2.25,1.9){$A$}
\rput(-1.15,0.75){$B$}
     \rput(0.05,-0.25){$O$}
\rput(2.25,1.9){$D$}
\rput(1.1,0.75){$C$}
\rput(0.5,1.14){$l$}
\rput(-3,-0.5){(a)}
\rput(0,0.5){$\bar \BBB$}
\rput(-0.1,2.2){\scalebox{1.76}{$\BBB$}}
}}
\scalebox{1}{
\rput(9.2,0.6){
\psline[arrows=->,arrowscale=1.5,linewidth=0.4pt](-0.3,0)(6,0)
\psline[arrows=->,arrowscale=1.5,linewidth=0.4pt](0,-0.3)(0,4)
\psline[linewidth=0.4pt](2,4)(0,0)(6,3)
     \psline[linecolor=blue,arrows=->,arrowscale=1.5,linewidth=0.8pt](1.7,3.4)(0.7,3.9)
          \rput(0.7,3.5){$\blue \eee_1$}
          \psline[linecolor=blue,arrows=->,arrowscale=1.5,linewidth=0.8pt](4.6,2.3)(5.1,1.3)
          \rput(4.8,1.2){$\blue \eee_2$}
                    \psline[linecolor=blue,arrows=->,arrowscale=1.5,linewidth=0.8pt](3.25,2.3)(3.65,3.3)
          \rput(4.1,3.05){$\blue -n_0$}
\psecurve[linewidth=1.2pt,linecolor=blue](2,3.7)(1.5,3)(1.25,2)(2,1.6)(3,1.7)(4,2)(4.75,2.5)
     \rput(1.55,2.75){$A$}
     \rput(4,1.7){$D$}
\rput(6,-0.25){$x$}
\rput(-0.25,4){$z$}
\rput(1.5,1.4){$l$}
     \rput(0.25,-0.25){$O$}
\psdots(1.5,3)(4,2)
       \psline[linewidth=0.3pt,linestyle=dashed,linecolor=blue](1.5,3)(4,2)
\rput(-0.7,-0.5){(b)}
\rput(0.9,0.9){$\bar \BBB$}
\rput(2.8,3){\scalebox{1.6}{$\BBB$}}
}}
\end{picture}
\caption{The angle $V$: $\, -z\cot\al \le x \le z\cot\bt$ and the curve $l$ are shown in the cases (a) $\al = \bt = \pi/4$, $\mu = \del_{-\pi/4} + \del_0 + \del_{\pi/4}$; (b) $\bt < \pi/2,\, \al > \pi/2$, $\mu$ is supported on $\Gam$. In the case (a) the curve $l$ is the broken line $ABCD$, and in the case (b) the curve $l$ is smooth.}
\label{figCurve}
\end{figure} 

Set $\int_{S^1} n\, d\mu(n) = c_0 n_0$, where $c_0 \ge 0$ and $n_0$ is a unit vector. Since spt$\,\mu = K$ is contained in $\Gam$, the integral $\int_{S^1} n\, d\mu(n)$ belongs to the cone formed by the rays through the points of $\Gam$ with the vertex at $O = (0,0)$. Further, $(\eee_1 + \eee_2, c_0 n_0) = \int_{S^1} (\eee_1 + \eee_2, n)\, d\mu(n) > 0$, since the integrand is a positive function. It follows that $c_0 > 0$ and $n_0 \in \Gam.$ Since spt$\,\mu$ contains both $\eee_1$ and $\eee_2$, we conclude that $n_0$ does not coincide with these vectors, that is, lies in the interior of $\Gam$.

It follows that $n_0$ can be represented as the positive linear combination $n_0 = \mu\eee_1 + \mu_2\eee_2$, $\mu_1 > 0$, $\mu_2 > 0$. Put the point $A$ on the ray $x = -z\cot\al,\, z \ge 0$ and the point $D$ on the ray $x = z\cot\bt,\, z \ge 0$, so as $|AO| = \mu_1 c_0$ and $|DO| = \mu_2 c_0$. The integral equation for the surface measure applied to the triangle $OAD$ takes the form $|AO| \eee_1 + |DO| \eee_2 + |AD| \eee_3 = \vec 0$, where $\eee_3$ is the outward normal to $AD$. From this equation one immediately obtains that $|AD| = c_0$ and $\eee_3 = -n_0$.

The center of mass of the auxiliary measure $\bar\mu = \mu + c_0 \del_{-n_0}$ is at the origin. Besides, spt$\,\mu$ contains $\eee_1$, $\eee_2$, and $-n_0$, and therefore, the linear span of spt$\,\mu$ is $\RRR^2$. Hence, according to Alexandrov's theorem, there exists a unique (up to a translation) planar convex body with the surface measure equal to $\bar\mu.$ The boundary of this body is the union of a curve inducing the measure $\mu$ and a line segment with the lenght $c_0$ and with the outward normal ${-n_0}$. Making if necessary a translation, we assume that this segment coincides with $AD$.

The examples of the resulting planar body are shown in Fig.~\ref{figCurve}. It is the quadrangle $ABCD$ in Fig.~\ref{figCurve}\,(a), and it is the figure bounded by the curve $l$ and by the line segment joining its endpoints in Fig.~\ref{figCurve}\,(b) .

Since the outward normals to $l$ lie in $\Gam$ and the outward normals at the endpoints of $l$ are $\eee_1$ and $\eee_2$, the curve is contained in the angle $V$ and touches the sides of the angle $OA$ and $OD$ at the endpoints of $l$. The resulting closed set $\BBB$ is bounded by the curve $l$ and by the sides of the angle $V$; it is convex and unbounded. The construction is done.

Let $\bar\BBB = V \setminus \BBB.$ For $y \in \RRR$ take the curve $y^2 l$ homothetic to $l$ with the center at the origin and with the ratio $y^2$ (if $y=0$, the curve degenerates to the point $O$), and denote by $\BBB_y$ the unbounded convex set bounded by the curve $y^2 l$ and by the rays. One can equivalently define $\BBB_y = V \setminus y^2 \bar\BBB.$ Thus, for $y \ne 0$ we have $\BBB_y = y^2\BBB$, and for $y=0$,  $\BBB_0 = V.$ One has the monotonicity relation $\BBB_{y_1} \subset \BBB_{y_2}$ for $|y_1| \ge |y_2|.$

Define $C$ as follows:
$$
C = \{ (x,y,z) : (x,z) \in \BBB_y \text{ and } z \le 1 \}.
$$
Loosely speaking, when cutting $C$ by planes $y =$ const, one obtains sets $\BBB_y \cap \{ z \le 1 \}.$

Notice that for $0 \le \lam \le 1$ and for all $y_1$ and $y_2$,
$$
\lam\BBB_{y_1} + (1-\lam)\BBB_{y_2} = \BBB_{\sqrt{\lam y_1^2 + (1-\lam) y_2^2}} \subset \BBB_{\lam y_1 + (1-\lam)y_2}.
$$
The inclusion in this formula is the direct consequence of the inequality $\sqrt{\lam y_1^2 + (1-\lam) y_2^2} \ge |\lam y_1 + (1-\lam)y_2|$. If both $y_1$ and $y_2$ are nonzero, the equality in this formula follows from the true identity for the convex set $\BBB$, $\lam y_1^2 \BBB + (1-\lam) y_2^2 \BBB = (\lam y_1^2 + (1-\lam) y_2^2) \BBB$. If $y_1 = y_2 = 0$, the equality takes the form $\lam V + (1-\lam) V = V$. If $y_1 \ne 0$ and $y_2 = 0$, the equality takes the form $\lam y_1^2 \BBB + (1-\lam) V = \lam y_1^2 \BBB.$

Let us show that $C$ is convex. To that end, take two points $\rrr_1 = (x_1, y_1, z_1)$ and $\rrr_2 = (x_2, y_2, z_2)$ in $C$ and show that $\lam \rrr_1 + (1-\lam) \rrr_2$ lies in $C$ for all $0 \le \lam \le 1$. Indeed, since $(x_1, z_1) \in \BBB_{y_1}$ and $(x_2, z_2) \in \BBB_{y_2}$, one has
$$
\lam (x_1, z_1) + (1-\lam) (x_2, z_2) \in \lam\BBB_{y_1} + (1-\lam)\BBB_{y_2} \subset \BBB_{\lam y_1 + (1-\lam)y_2}.
$$
Besides, one obviously has $\lam z_1 + (1-\lam) z_2 \le 1$. It follows that $\lam \rrr_1 + (1-\lam) \rrr_2 \in C$.

One easily sees that each point $(x,0,z)$, with $(x,z)$ in the interior of $V \cap \{ z \le 1 \}$, is an interior point of $C$. Further, for $|y|$ sufficiently large, $\BBB_y \cap \{ z \le 1 \} = \emptyset$, hence $C$ is bounded.

Now check that $C$ is closed. Let a sequence $\{ \rrr_i = (x_i, y_i, z_i) \} \subset C$ converge to a point $\rrr = (x,y,z).$ If $y \ne 0$ then $\frac{1}{y_i^2} (x_i, z_i) \in \BBB$ and since $\BBB$ is closed, the limiting point $\frac{1}{y^2} (x,z)$ also lies in $\BBB$. It follows that $(x,y,z) \in C.$ If $y = 0$ then $(x_i, z_i) \in V \Rightarrow (x,z) \in V$, hence $(x,0,z) \in C$. Thus, $C$ is a convex body.

Take $\vphi \in (\bt,\, \pi - \al)$ and draw the ray in the $xz$-plane with the vertex at $O$ and with the director vector $(\cos\vphi, \sin\vphi)$. This ray intersects $\pl\BBB$ at a point $\tau(\cos\vphi, \sin\vphi)$, $\tau > 0$. Thus, the intersection of $\pl C$ and the plane through $\rrr_0 = (0,0,0)$ with the normal $(\sin\vphi, 0, -\cos\vphi)$ is the curve $(y^2 \tau\cos\vphi, y, y^2 \tau\sin\vphi)$, $|y| \le 1/\sqrt{\tau\sin\vphi}$. It follows that the $y$-axis touches $C$, and therefore, is contained in any plane of support through $\rrr_0 = (0,0,0)$.

Further, $C$ is contained in the dihedral angle \eqref{dih2}, but is not contained in any smaller dihedral angle with the same edge. It follows that the tangent cone at $\rrr_0$ is the dihedral angle \eqref{dih2}.

The intersection of $\BBB$ with the line $z = \tau$ for $\tau \ge 0$ is either empty, or a point, or a line segment. Let its length be $L(\tau)$; it is a non-negative continuous function equal to zero when $\tau$ is sufficiently small. Further, the set $B_t$ for $0 < t \le 1$ is
$$
B_t = C \cap \{ z = t \} = \{ (x,y,t) : (x,t) \in \BBB_y \}.
$$
The intersection of $B_t$ with each plane $y = \text{const} \ne 0$ is a segment (maybe degenerating to a singleton) or the empty set, and its length equals $y^2 L(t/y^2).$ Hence the area of $B_t$ is
$$
|B_t| = \int_\RRR y^2 L(t/y^2)\, dy.
$$
This integral is finite, since the integrand is zero outside a bounded segment. Making the change of variable $\xi = y/\sqrt t$, one obtains
$$
|B_t| = b\, t^{3/2}, \quad \text{where} \quad b = \int_\RRR \xi^2 L(1/\xi^2)\, d\xi > 0.
$$

The set $S_t$ for $0 < t \le 1$ is
$$
S_t = \pl C \cap \{ z \le t \} = \{ (x,y,z) : (x,z) \in \pl\BBB_y \text{ and } z \le t \}.
$$
It is the disjoint union, $S_t = S_t^+ \cup S_t^- \cup S_t^0$, where
$$
S_t^+ = S_t \cap \{ x = z\cot\bt \}, \quad S_t^- = S_t \cap \{ x = -z\cot\al \}, \quad S_t^0 = S_t \cap \{ -z\cot\al < x < z\cot\bt \};
$$
that is, $S_t^+$ and $S_t^-$ are the parts of $S_t$ contained in the faces of the dihedral angle \eqref{dih2} and $S_t^0$ is the part of $S_t$ contained in the interior of the angle.

For $\tau \ge 0$ the part of $\pl\BBB$ below the line $z = \tau$ is the union of three curves; the first and the third ones are line segments contained in the sides of the angle $V$, that is, in the rays $x = z\cot\bt$, $z \ge 0$ and $x = -z\cot\al$, $z \ge 0$, respectively, and the second curve lies between the sides of $V$ and is contained in $l$. Denote by $L_+(\tau),\, L_-(\tau),\, L_0(\tau)$ the lengths of these curves. They are non-negative continuous functions equal to zero for $\tau$ sufficiently small, and $L_-(\tau) \le \tau/\sin\al$, $L_+(\tau) \le \tau/\sin\bt$.

The planar surfaces $S_t^+$ and $S_t^-$ (with the points of the form $(x,0,z)$ taken away) are composed of points $(x,y,z)$ such that $(x/y^2,\, z/y^2)$ lies on the intersection of $\pl\BBB$ with the rays $x = z\cot\bt$ and $x = -z\cot\al$, respectively, and $z \le t$. Their areas are equal to
$$
|S_t^\pm| = \int_\RRR y^2 L_\pm(t/y^2)\, dy,
$$
and making the change of variable $\xi = y/\sqrt t$, one obtains
$$
|S_t^\pm| = s_\pm\, t^{3/2}, \quad \text{where} \quad s_\pm = \int_\RRR \xi^2 L_\pm(1/\xi^2)\, d\xi > 0.
$$

The outward normals to the surfaces $S_t^-$ and $S_t^+$ are $\eee_1$ and $\eee_2$, hence the normalized measures induced by $S_t^-$ and $S_t^+$ are the atoms $\frac{s_-}{b}\, \del_{e_1}$ and $\frac{s_+}{b}\, \del_{e_2}$. Denote by $\nu_t^0$ the normalized measure induced by the surface $S_t^0$; then we have
$$
\nu_t = \frac{s_-}{b}\, \del_{e_1} + \frac{s_+}{b}\, \del_{e_2} + \nu_t^0.
$$
It remains to describe $\nu_t^0$.

Denote by $l_0$ the part of the curve $l$ that lies in the interior of the angle $V$, $-z\cot\al < x < z\cot\bt$, and consider the natural parameterization of $l_0$, $(x(s),z(s)),\, s \in (0,\, s_0)$, starting from the ray $x = -z\cot\al$. The length $s_0$ of the curve $l_0$ is finite, and the derivative $(x'(s),z'(s))$ exists for almost all $s$ and satisfies $x'^2 + z'^2 = 1.$ The surface $S_t^0$ is composed of points $(x,y,z)$ such that $y \ne 0$ and $(x/y^2,\, z/y^2) \in l_0$.

Parameterize $S_t^0$ by the mapping $(s,y) \mapsto \rrr(s,y) = (y^2 x(s), y, y^2 z(s))$ defined on the domain $D_t = \{ (s,y) : 0 \le s \le s_0,\, 0 < |y| \le \sqrt{t/z(s)} \}$. This mapping is injective and $D_t$ is bounded.

The element of area of the surface $S_t^0$ is $|\rrr'_s \times r'_y| ds dy$, and the outward normal to $S_t^0$ is $\rrr'_s \times r'_y/|\rrr'_s \times r'_y|$. One easily finds that
$$\rrr'_s(s,y) \times r'_y(s,y) = y^2 (z'(s),\, -2y\Del(s),\, -x'(s)), \quad \text{where} \ \ \Del(s) =
\left|\! \begin{array}{cc} x(s) & z(s)\\ x'(s) & z'(s) \end{array} \!\right|,
$$
and $|\rrr'_s(s,y) \times r'_y(s,y)| = y^2 \sqrt{1 + 4y^2 \Del(s)^2}$. Note that the vector $(x', z')$ is unitary and $|(x, z)| \le C$, where $C$ is the maximum distance between the points of $l_0$ and the origin, therefore the function $\Del(s)$ is bounded, $|\Del(s)| \le C$.

We have found that the element of area of $S_t^0$ is $y^2 \sqrt{1 + 4y^2 \Del(s)^2}\, ds dy$, and the outward normal to $S_t^0$ at $\rrr(s,y)$ is
$$
n(s,y) = (z'(s), -2y\Del(s), -x'(s))/\sqrt{1 + 4y^2\Del(s)^2}.
$$
Let us now define the auxiliary normalized measure $\tilde\nu_t$ induced by the family of vectors $\tilde n(s,y) = (z'(s), 0, -x'(s))$ on the measurable space $D_t$ with the measure $d\tilde\mu(s,y) = y^2 ds dy$ as follows: for any Borel set $A \subset S^2$,
$$
\tilde\nu_t(A) = \frac{1}{b t^{3/2}}\, \tilde\mu\big( \{ (s,y) : \tilde n(s,y) \in A \} \big).
$$

The measure $\tilde\nu_t$ is easier to study than $\nu_t^0$; below we show that it actually does not depend on $t$. Let us show that $\tilde\nu_t$ and $\nu_t^0$ are asymptotically equivalent: they either both converge to the same limit, or both diverge. To that end it suffices to show that for any continuous function $f$ on $S^2$,
$$
\int_{S^2} f(n)\, d(\nu_t^0 - \tilde\nu_t)(n) \to 0 \quad \text{as} \ \ t \to 0.
$$
This integral can be rewritten in the form
$$
\frac{1}{b t^{3/2}} \int\!\!\!\int_{D_t} f\Big( \frac{(z'(s), -2y\Del(s), -x'(s))}{\sqrt{1 + 4y^2\Del(s)^2}} \Big)\, y^2 \sqrt{1 + 4y^2\Del(s)^2}\, ds dy - 
$$ $$
- \frac{1}{b t^{3/2}} \int\!\!\!\int_{D_t} f(z'(s), 0, -x'(s))\, y^2\, ds dy $$
$$
=\frac{1}{b t^{3/2}} \int\!\!\!\int_{D_t} \Big[ f\Big( \frac{(z'(s), -2y\Del(s), -x'(s))}{\sqrt{1 + 4y^2\Del(s)^2}} \Big)\, \sqrt{1 + 4y^2\Del(s)^2} -
f(z'(s), 0, -x'(s)) \Big] y^2 \, ds dy.
$$
Since the function $f$ is continuous, and therefore, uniformly continuous on $S^2$ and the function $\Del(s)$ is bounded, one concludes that the expression in the square brackets in the latter integral goes to zero as $y \to 0$ uniformly in $s$, and therefore, is smaller than a positive monotone decreasing function $\gam(y)$ going to zero as $y \to 0$. Making the change of variable $\xi = y/\sqrt t$, one obtains that the integral is smaller than
$$
\frac{1}{b t^{3/2}} \int\!\!\!\int_{D_t} y^2 \gam(y)\, ds dy = \frac{1}{b t^{3/2}} \int_0^{s_0} ds \int_{-\sqrt{t/z(s)}}^{\sqrt{t/z(s)}} y^2 \gam(y)\, dy =
\frac{1}{b} \int_0^{s_0} ds \int_{-1/\sqrt{z(s)}}^{1/\sqrt{z(s)}} \xi^2 \gam(\sqrt t\, \xi)\, d\xi.
$$
Since the integrand $\xi^2 \gam(\sqrt t\, \xi)$ is monotone decreasing to zero as $t \to 0$, we conclude that the integral in the right hand side of this expression goes to 0 as $t \to 0$.

It remains to study the measure $\tilde\nu_t$. It is supported on $\Gam$, since all vectors $(z'(s), 0, -x'(s))$ belong to $\Gam.$ Recall that we identify the value $\vphi$ with the vector $(\sin\vphi, 0, -\cos\vphi)$.

Take an interval $\UUU = (\vphi_1,\, \vphi_2) \subset (-\al,\, \bt).$ One has
$$
\tilde\nu_t(\UUU) = \frac{1}{b t^{3/2}} \int_{F_\mu(\vphi_1)}^{F_\mu(\vphi_2-0)} ds \int_{-\sqrt{t/z(s)}}^{\sqrt{t/z(s)}} y^2\, dy =
\frac{1}{b} \int_{F_\mu(\vphi_1)}^{F_\mu(\vphi_2-0)} \frac{2}{3}\, \frac{1}{z(s)^{3/2}}\, ds.
$$
It follows that $\tilde\nu_t$ does not depend on $t$, $\tilde\nu_t = \tilde\nu$, and hence, $\nu_t^0 \to \tilde\nu$ as $t \to 0$. If $\UUU \cap K = \emptyset$ then $F_\mu(\vphi_1) = F_\mu(\vphi_2-0)$ and, according to the above formula, $\tilde\nu(\UUU) = 0$. If, otherwise, $\UUU \cap K \ne \emptyset$  then $F_\mu(\vphi_1) < F_\mu(\vphi_2-0)$, and therefore, $\tilde\nu(\UUU) > 0$. Thus, spt$\,\tilde\nu \setminus \{ -\al,\, \bt \}$ coincides with $K \setminus \{ -\al,\, \bt \}$. It remains to note that $\nu_t$ converges to $\nu_* =  \frac{s_-}{b}\, \del_{e_1} + \frac{s_+}{b}\, \del_{e_2} + \tilde\nu$ and spt$\,\nu_* = K.$
This finishes the proof of Theorem \ref{t1b}.





\section{Application to Newton's problem for convex bodies}\label{sect t3}

\subsection{Newton's problem in terms of surface measure}

It is useful to represent the integral in \eqref{resN} in a different form. Consider the convex body $C = C_{(u)}$ associated with the function $u$,\, $C = \{ (x,y,z) \in \RRR^3 : (x,y) \in \Om,\ 0 \le z \le u(x,y) \}$. Its boundary is $\pl C = \text{graph}(u) \cup (\Om \times \{0\})$. The outer normal to $C$ is $\nnn_\xi = (-\nabla u(x,y),\, 1)/\sqrt{1 + |\nabla u(x,y)|^2}$ at a point $\xi = (x,y,u(x,y)) \in \text{graph}(u)$, and $\nnn_\xi = (0,0,-1)$ at $\xi \in \Om \times \{0\}.$ The integral can then be represented as the surface integral
$$
F(u) = \int\!\!\!\int_{\text{graph}(u)} f(n_\xi)\, d\HHH^2(\xi),
$$
where the function $f : S^2 \to \RRR$ is defined by $f(n_1,n_2,n_3) = (n_3)^3$ and $\HHH^2$ is the 2-dimensional Hausdorff measure. Indeed, the integrand in \eqref{resN} is $1/(1 + |\nabla u(x,y)|^2) = ((\nnn_\xi)_3)^2$, and the slope of graph$(u)$ at $\xi= (x,y,u(x,y))$ is $(\nnn_\xi)_3 = 1/\sqrt{1 + |\nabla u(x,y)|^2}$, therefore the element of the surface area $d\HHH^2(\xi)$ on graph$(u)$ and the element of area $dx dy$ on the horizontal plane are related as follows: $dx dy = (\nnn_\xi)_3\, d\HHH^2(\xi).$

Taking into account that
$$
\int\!\!\!\int_{\Om \times \{0\}} f(n_x)\, d\HHH^2(\xi) = \int\!\!\!\int_{\Om \times \{0\}} (-1)\, d\HHH^2(\xi) = -|\Om| = -\pi,
$$
one can write
$$
F(u) = \int\!\!\!\int_{\pl C} f(n_\xi)\, d\HHH^2(\xi) + \pi.
$$

This integral representation in Newton's problem (in a more general setting) was first used in the paper by Buttazzo and Guasoni \cite{BG97}.

Further, making the change of variable induced by the map $\xi \mapsto n_\xi$ from $\pl' C$ to $S^2$, we obtain the following representation
$$
F(u) = \int_{S^2} f(\nnn)\, d\nu_C(\nnn) + \pi,
$$
where $\nu_C$ means the surface measure of $C$.

\subsection{2-dimensional problems of minimal resistance}


Buttazzo and Kawohl in \cite{BK} state the 2D analogue of Newton's problem: minimize $\int_{-^\rrr}^{\rrr} (1 + u'^2(x))^{-1} dx$ in the class of concave functions $u : [-\rrr,\, \rrr] \to [0,\, M]$, and give the explicit solution: $u(x) = \min\{ M,\, \rrr - |x| \}$, if $0 < M < \rrr$, and $u(x) = M(1 - |x|/\rrr)$, if $M \ge \rrr$.

We will need the following modification of the above problem:
\beq\label{pr00}
\hspace*{-56mm} \text{Minimize the integral } \qquad \int_0^1 \frac{1}{1 + u'(x)^2}\, dx
\eeq
in the class $\UUU_M$ of concave continuous functions $u : [0,\, 1] \to [0,\, M]$ satisfying the condition $u(0) = M$. The solution $u_*$ to this problem is the following: if $0 < M < 1$ then $u_*(x) = \min\{ M,\, 1-x \}$, and if $M \ge 1$ then $u_*(x) = M(1-x)$ (see Fig.~\ref{fig2D}).
The proof is a direct consequence of the solution of the above problem and is left to the reader.

         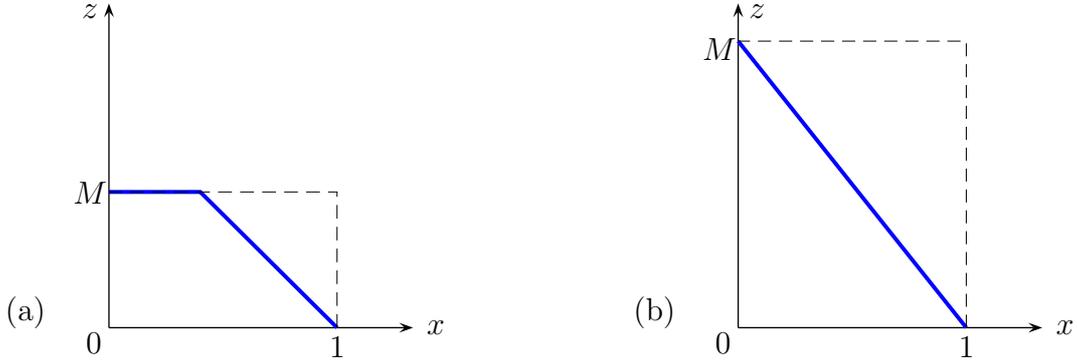
\begin{figure}[h]
\begin{picture}(0,120)
\scalebox{1}{
\rput(2,0.2){
\psline[arrows=->,arrowscale=1.5,linewidth=0.5pt](0,0)(4,0)
\psline[arrows=->,arrowscale=1.5,linewidth=0.5pt](0,0)(0,4.3)
\psline[linecolor=blue,linewidth=1.5pt](0,1.8)(1.2,1.8)(3,0)
\psline[linestyle=dashed,linewidth=0.3pt](3,0)(3,1.8)(0,1.8)
\rput(-0.25,1.8){$M$}
\rput(-0.2,-0.2){0}
\rput(3,-0.25){1}
\rput(4.3,0){$x$}
\rput(-0.25,4.2){$z$}
\rput(-1.1,0.2){(a)}
}}
\scalebox{1}{
\rput(10,.2){
\psline[arrows=->,arrowscale=1.5,linewidth=0.5pt](0,0)(4,0)
\psline[arrows=->,arrowscale=1.5,linewidth=0.5pt](0,0)(0,4.3)
\psline[linecolor=blue,linewidth=1.5pt](0,3.8)(3,0)
\psline[linestyle=dashed,linewidth=0.3pt](3,0)(3,3.8)(0,3.8)
\rput(-0.25,3.7){$M$}
\rput(-0.2,-0.2){0}
\rput(3,-0.25){1}
\rput(4.3,0){$x$}
\rput(0.25,4.2){$z$}
\rput(-1.1,0.2){(b)}
}}

\end{picture}
\caption{The solution $z = u_*(x)$ to the 2-dimensional problem (a) for $0 < M < 1$; (b) for $M \ge 1$.}
\label{fig2D}
\end{figure} 

Given a function $u$ from the class $\UUU_M$, denote by $C = C_{(u)}$ the convex body associated with $u$ on the $xz$-plane,\, $C = \{ (x,z): 0\le x \le 1,\ 0 \le z \le u(x) \}$. The surface measure of $C$ is $\nu_C = M \del_{(-1,0)} + \del_{(0,-1)} + \nu_u$, where the measure $\nu_u$ is induced by the family of outward normals to the curve $\text{graph}(u) \cup \big( \{1\} \times [0,\, u(1)] \big)$. More precisely, if a Borel set $A \subset S^1$ does not contain $(1,0)$ then $\nu_u(A)$ is equal to the linear measure of the set $\{ (x, u(x)) : \frac{(-u'(x)), 1)}{\sqrt{1 + u'^2(x)}} \in A \}$, and $\nu_u((1,0)) = u(1)$. Since $u$ is monotone decreasing, the measure $\nu = \nu_u$ is supported in the first quarter of the unit circumference,
$$
{\rm (i)} \quad \text{spt}\,\nu \subset S^1 \cap \{ x \ge 0,\, z \ge 0 \}. \hspace*{39mm}
$$
Besides, $\nu = \nu_u$ satisfies the relation
$$
{\rm (ii)} \quad \int_{S^1} \nnn\, d\nu(\nnn) = (M, 1). \hspace*{50mm}
$$

Inversely, by Alexandrov's theorem, given a measure $\nu$ on $S^1$ satisfying (i) and (ii), there exists a unique, up to translations, planar convex body whose surface measure is $M \del_{(-1,0)} + \del_{(0,-1)} + \nu$. The lower and the left parts of its boundary are formed by a horizontal segment of length 1 and a vertical segment of length $M$. Making if necessary a translation, one can assume that these segments are $[0,\, 1] \times \{ 0 \}$ and $\{ 0 \} \times [0,\, M]$. The resting part of the boundary is the union of the graph of a concave continuous non-negative monotone decreasing function $u$ joining the points $(0,M)$ and $(1,u(1))$ and the segment $\{ 1 \} \times [0,\, u(1)]$ (degenerating to the point $(1,0)$ if $u(1) = 0$).

Thus, there is a one-to-one correspondence between the class of measures satisfying (i) and (ii) and the class of functions $\UUU_M$.

The integral in \eqref{pr00} can be written in terms of measures in the form
 $ \int_{S^1} f(\nnn)\, d\nu_u(\nnn)$
 where, by slightly abusing the language, we denote by $f$ the function on $S^1$ defined by $f(n_1, n_3) = (n_3)^3$. Thus, problem \eqref{pr00} can be written in the equivalent form as follows:
 \beq\label{pr01}
\hspace*{-56mm} \text{Minimize the integral } \qquad \int_{S^1} f(\nnn)\, d\nu(\nnn)
\eeq
in the class of measures $\nu$ satisfying (i) and (ii). The unique solution to this problem is
 $
 \nu_* = (1-M) \del_{(0,1)} + \sqrt 2\, M \del_{\frac{1}{\sqrt{2}}(1,1)},
 $
 if $0 < M < 1$, and
 $
 \nu_* = \sqrt{1 + M^2}\, \del_{\frac{1}{\sqrt{1+M^2}}(M,1)},
 $
 if $M \ge 1$.

Let us formulate separately this statement in the particular case $M=1$, since it will be used below in this section.

\begin{utv}\label{utv meas}
The minimum of the integral $\int_{S^1} f(n)\, d\nu(n)$ in the class of measures $\nu$ on $S^1$ satisfying the conditions

(i) spt$\,\nu$ lies in the quarter of the circumference $x^2 + z^2 = 1,\, x \ge 0,\, z \ge 0$;

(ii) $\int_{S^1} n\, d\nu(n) = \frac{1}{\sqrt{2}}(1, 1)$\\
equals $1/(2\sqrt 2)$, and the unique minimizer is the atomic measure $\del_{\frac{1}{\sqrt{2}}(1,1)}$.
\end{utv}

\subsection{Proof of Theorem \ref{t3}}

Let the function $u = u_*$ minimize the integral \eqref{resN}. Consider the associated convex body $C = C_{(u)} = \{ (x,y,z) \in \RRR^3 : (x,y) \in \Om,\ 0 \le z \le u(x,y) \}$. Note that almost all points of the curve $\pl L_M$ are regular, and take a regular point $(\bar x, \bar y) \in \pl L_M$. The tangent cone to $C$ at $(\bar x, \bar y, M)$ is a dihedral angle. One of its faces is contained in the horizontal plane $z = M$. Let $-k$ ($k > 0$) be the slope of the other face.

For any sequence of regular points $(x_i, y_i)$ converging to $(\bar x, \bar y)$, each partial limit of the corresponding sequence $|\nabla u(x_i, y_i)|$ lies in the interval $[0,\, k]$. Since $|\nabla u(x,y)| \ge 1$ for $(x,y) \in \Om \setminus L_M$, one concludes that for a sequence $(x_i, y_i) \subset \Om \setminus L_M$ converging to $(\bar x, \bar y)$, each partial limit of $|\nabla u(x_i, y_i)|$ lies in the interval $[1,\, k]$, and therefore, $k \ge 1$.

We are going to prove that $k = 1$. It will follow that each partial limit of the sequence $|\nabla u(x_i, y_i)|$ is equal to 1, and therefore, the sequence converges to 1. The theorem will be proved.

Suppose that $k > 1$, that is, the angle of slope of one of faces of the dihedral angle is greater than $45^0$. We are going to come to a contradiction with optimality of $u$.

Let $\rrr_0 = (\bar x, \bar y, M)$, and let $\eee_1$ and $\eee_2$ be the outward normals to the faces of the dihedral angle. We have $\eee_1 = (0,0,1)$ and without loss of generality assume that $\eee_2 = \frac{1}{\sqrt{1+k^2}}\, (k, 0, 1)$; it suffices to choose a coordinate system so as the projection of $\eee_2$ on the horizontal plane is directed along the $x$-axis. Take $\eee = \frac{1}{\sqrt{2}}\, (1, 0, 1)$; then we have $\eee = \lam_1 \eee_1 + \lam_2 \eee_2$ with $\lam_1 = (k - 1)/(\sqrt{2}\, k)$ and $\lam_2 = \sqrt{1+k^2}/(\sqrt{2}\, k)$.

For $t > 0$ consider the convex body $C^{(t)}$ obtained by cutting off a part of $C$ containing $(\bar x, \bar y, M)$ by the plane $(\rrr-\rrr_0, \eee) =- t$. This plane is the graph of the linear function $z = M - \sqrt 2\, t + \bar x -x$. The upper part of the surface of $C^{(t)}$ is the concave function $u^{(t)}$ defined by $u^{(t)}(x,y) = \min \{ u(x,y),\, M - \sqrt 2\, t + \bar x -x \}$. For $t$ sufficiently small the plane lies above the disc $\Om \times \{0\}$, and therefore, the non-negative function $u^{(t)}$ is defined on all of $\Om$.

Let $B_t$ be the intersection of $C$ with the cutting plane and $S_t$ be the part of $\pl C$ located above the plane.
One has $B_t = \{ (x,y,z) : 0 \le z = M - \sqrt 2\, t + \bar x -x \le u(x,y) \}$ and  $S_t = \{ (x,y,z) : u(x,y) = z \ge M - \sqrt 2\, t + \bar x -x \}$.
The parts of the surfaces $\pl C$ and $\pl C^{(t)}$ situated in the upper half-space $z \ge M - \sqrt 2\, t + \bar x -x$ are, respectively, $S_t$ and $B_t$. The parts of $\pl C$ and $\pl C^{(t)}$ in the lower half-space $z < M - \sqrt 2\, t + \bar x -x$ coincide. Hence the difference $F(u) - F(u^{(t)})$ can be written in terms of surface integrals as
$$
F(u) - F(u^{(t)}) = \int\!\!\!\int_{S_t} f(n_\xi)\, d\HHH^2(\xi) - \int\!\!\!\int_{B_t} f(n_\xi)\, d\HHH^2(\xi);
$$
recall that $n_\xi$ means the outward normal to the corresponding surface ($S_t$ or $B_t$) and the function $f : S^2 \to \RRR$ is given by $f(n_1,n_2,n_3) = (n_3)^3$.

Since the outward normal to $B_t$ is $\eee = \frac{1}{\sqrt{2}}\, (1,0,1)$ and $f(\eee) = 1/(2\sqrt{2})$,  the latter integral in the right hand side equals $|B_t|/(2\sqrt{2})$. Rewriting the former integral in terms of the normalized measure induced by $S_t$, one obtains
$$
\frac{1}{|B_t|}\big( F(u) - F(u^{(t)}) \big) = \int_{S^2} f(n)\, d\nu_t(n) - \frac{1}{2\sqrt{2}}.
$$

By Theorem \ref{t1}, there exists a partial limit $\nu_* = \lim_{i\to\infty} \nu_{t_i}$, and the support of $\nu_*$ is contained in the arc of the circumference $x^2 + z^2 = 1,\, z = 0$ bounded by the vectors $\eee_1$ and $\eee_2$ and contains these vectors. Note that each measure $\nu = \nu_t$ satisfies the condition $\int_{S^2} \nnn\, d\nu(\nnn) = \eee$, therefore the partial limit $\nu = \nu_*$ also satisfies this condition. We have
\beq\label{expression}
\lim_{i\to\infty}\frac{1}{|B_{t_i}|} \big( F(u) - F(u^{(t_i)}) \big) =  \int_{S^2} f(n)\, d\nu_*(n) - \frac{1}{2\sqrt{2}}.
\eeq

Again, slightly abusing the language, below we will denote by $f$ the function on $S^1$ defined by $f(n_1, n_3) = (n_3)^3$ and by $\nu_*$ the push-forward measure of $\nu_*$ by the natural projection $(x,y,z) \mapsto (x,z)$ from $S^2$ to $\RRR^2$. We have

(i) spt$\,\nu_*$ is contained in the quarter of the circumference $x^2 + z^2 = 1,\, x \ge 0,\, z \ge 0$;

(ii) $\int_{S^1} n\, d\nu_*(n) = \frac{1}{\sqrt{2}}(1, 1)$. \\
The expression in the right hand side of \eqref{expression} takes the form
\beq\label{expr}
\int_{S^1} f(n)\, d\nu_*(n) - \frac{1}{2\sqrt{2}}.
\eeq

According to Proposition \ref{utv meas}, the infimum of this expression in the set of measures $\nu$ satisfying (i) and (ii)  is attained at the atomic measure $\del_{\frac{1}{\sqrt{2}}(1,-1)}$ and is equal to 0. The measure $\nu_*$ does not coincide with the minimizer, since its support contains the points $(0,1)$ and $\frac{1}{\sqrt{1+k^2}}\, (k, 1)$, and therefore the expression \eqref{expr} is strictly greater than zero. It follows that for some $t = t_i$, $F(u) > F(u^{(t_i)})$, hence $u$ is not optimal. Theorem \ref{t3} is proved.

\section*{Acknowledgements}

{This work was supported by Foundation for Science and Technology (FCT), within project UID/MAT/04106/2019 (CIDMA). I am deeply grateful to G. Buttazzo, G. Wachsmuth, and V. Alexandrov for useful discussions of the issues considered in this paper.


\begin{thebibliography}{99}

\bibitem{Alex}
A. D. Alexandrov.\, {\it Selected Works. Part I: Selected Scientific Papers}, { Chapter V, \S 3.}\, Ed. by Yu.~G. Reshetnyak and S.~S. Kutateladze. Gordon and Breach Publishers (1996). 

\bibitem{BelloniKawohl}
M. Belloni and B. Kawohl. \textit{A paper of Legendre revisited}. Forum Math. {\bf 9}, 655-668 (1997).

\bibitem{BW prescribed volume}
M. Belloni and A. Wagner. {\it Newton’s problem of minimal resistance in the class of bodies with prescribed volume}.
J. Convex Anal. {\bf 10}, 491–500 (2003).

\bibitem{BrFK}
F. Brock, V. Ferone and B. Kawohl.\, \textit{A symmetry problem in the calculus of variations}.\, Calc. Var. {\bf 4}, 593-599 (1996).

\bibitem{BFK}
G. Buttazzo, V. Ferone, B. Kawohl.\,
\textit{Minimum problems over sets of concave functions and related questions}.\, Math. Nachr. {\bf 173}, 71--89 (1995).

\bibitem{BK}
G. Buttazzo, B. Kawohl.\, \textit{On Newton's problem of minimal resistance}.\, Math. Intell. {\bf 15}, 7--12 (1993).

\bibitem{BG97}
G. Buttazzo, P. Guasoni.\, {\it Shape optimization problems over classes of convex domains}.\, J. Convex Anal. {\bf 4}, No.2, 343-351 (1997).

\bibitem{CL1}
M. Comte, T. Lachand-Robert.\, \textit{Newton's problem of the body of minimal resistance under a single-impact assumption}.\,
Calc. Var. Partial Differ. Equ. {\bf 12}, 173-211 (2001).

\bibitem{CL2}
M. Comte, T. Lachand-Robert.\, \textit{Existence of minimizers for Newton's problem of the body of minimal resistance under a single-impact assumption}.\, J. Anal. Math. {\bf 83}, 313-335 (2001).

\bibitem{LO}
T. Lachand-Robert and E. Oudet.\, \textit{Minimizing within convex bodies using a convex hull method}.\, SIAM J. Optim. {\bf 16}, 368-379 (2006).

\bibitem{LP1}
T. Lachand-Robert, M.~A. Peletier.\,
\textit{Newton's problem of the body of minimal resistance in the class of convex developable functions}.\, Math. Nachr. {\bf 226}, 153-176 (2001).


\bibitem{N}
I. Newton. {\it Philosophiae naturalis principia mathematica}. (London: Streater) 1687.

\bibitem{ARMA}
A. Plakhov. {\it Billiards and two-dimensional problems of optimal resistance}. Arch. Ration. Mech. Anal. {\bf 194}, 349-382 (2009). 

\bibitem{optimal roughening}
A Plakhov. {\it Optimal roughening of convex bodies}. Canad. J. Math. {\bf 64}, 1058-1074 (2012). 

\bibitem{bookP}
A. Plakhov. {\it Exterior billiards. Systems with impacts outside bounded domains}. Springer, New York, 2012. xiv+284 pp. ISBN: 978-1-4614-4480-0

\bibitem{PT thermal}
A. Plakhov and D. Torres.\, {\it Newton's aerodynamic problem in media of chaotically moving particles}. Sbornik: Math. {\bf 196}, 885-933 (2005).

\bibitem{PlakhovSIA}
A. Plakhov. {\it Newton’s problem of minimal resistance under the single-impact assumption}. Nonlinearity {\bf 29}, 465-488 (2016).

\bibitem{rough2D}
A. Plakhov. {\it Billiard scattering on rough sets: Two-dimensional case}. SIAM J. Math. Anal. {\bf 40}, 2155-2178 (2009). 

\bibitem{Nonlinearity}
A. Plakhov and P. Gouveia. {\it Problems of maximal mean resistance on the plane}. Nonlinearity {\bf 20}, 2271-2287 (2007).

\bibitem{Magnus}
A. Plakhov and T. Tchemisova. {\it Force acting on a spinning rough disk in a flow of non-interacting particles}.
Doklady Math. {\bf 79}, 132-135 (2009).

\bibitem{Pogorelov}
A.\,V. Pogorelov. {\it Extrinsic geometry of convex surfaces.} Providence, R.I.: American Mathematical Society (AMS). (1973).

\bibitem{W}
G. Wachsmuth.\, {\it The numerical solution of Newton’s problem of least resistance}.\, Math. Program. A {\bf 147}, 331-350 (2014).



\end{thebibliography}
\end{document}